\documentclass[a4paper,reqno,12pt]{amsart}
\usepackage{fancyhdr}
\usepackage{amsmath}
\usepackage{dsfont}
\usepackage{hyperref}
\usepackage[mathscr]{eucal}
\usepackage[cp1251]{inputenc}
\usepackage[english]{babel}
\usepackage{cite,enumerate,float,indentfirst}
\usepackage{graphicx}
\usepackage{xcolor}
\usepackage{latexsym,a4,mathrsfs,amsthm,amsmath,amssymb,url}
\usepackage{amsfonts}
\usepackage{amssymb}
\usepackage{epstopdf}

\newcommand\NoBlackBoxes{\global\overfullrule0pt}
\NoBlackBoxes
\numberwithin{equation}{section}

\newtheorem{theorem}{Theorem}[section]
\newtheorem{lemma}[theorem]{Lemma}

\newtheorem{corollary}[theorem]{Corollary}
\theoremstyle{remark}
\newtheorem*{remark}{Remark}

\newcommand{\R}{\mathbb{R}}
\newcommand{\C}{\mathbb{C}}
\newcommand{\T}{\mathbb{T}}

\newcommand{\J}{\mathbb{J}}

\newcommand{\Cond}{{\bf (C0)}}
\newcommand{\CondTwo}{{\bf (C1)}}

\newcommand{\X}{{\bf X}}

\newcommand{\A}{{\bf A}}

\newcommand{\I}{{\bf I}}
\newcommand{\RR}{{\bf R}}

\newcommand{\W}{{\bf W}}

\newcommand{\ee}{{\bf e}}

\newcommand{\vect}[1]{\boldsymbol{#1}}

\DeclareMathOperator{\Tr}{Tr}

\DeclareMathOperator{\E}{\mathbb{E}}

\DeclareMathOperator{\Var}{\mathbb{D}}
\DeclareMathOperator{\Pb}{\mathbb{P}}
\DeclareMathOperator{\one}{\mathds{1}}

\DeclareMathOperator{\imag}{Im}

\begin{document}

\vspace{1in}

\title[Local Semicircle Law]{\bf Local Semicircle law Under Moment conditions. \\ Part II: Localization and Delocalization.}

%\vspace{2in}
\author[F. G{\"o}tze]{F. G{\"o}tze}
\address{Friedrich G{\"o}tze\\
 Faculty of Mathematics\\
 Bielefeld University \\
 Bielefeld, Germany
}
\email{goetze@math.uni-bielefeld.de}

\author[A. Naumov]{A. Naumov}
\address{Alexey A. Naumov\\
 Faculty of Computational Mathematics and Cybernetics\\
 Lomonosov Moscow State University \\
 Moscow, Russia \\
 and Institute for Information Transmission Problems of the Russian Academy of Sciences (Kharkevich Institute), Moscow, Russia\\
 and Chinese University of Hong Kong, Department of Statistics, Hong Kong
 }
\email{anaumov@cs.msu.su}

\author[A. N. Tikhomirov]{A. N. Tikhomirov}
\address{Alexander N. Tikhomirov\\
 Department of Mathematics\\
 Komi Research Center of Ural Division of RAS \\
 Syktyvkar, Russia
 }
\email{tikhomirov@dm.komisc.ru}

\thanks{All authors were supported by CRC 701 “Spectral Structures and Topological Methods in
Mathematics”. A.~Naumov and A.~Tikhomirov were supported by RFBR N~14-01-00500. A. Tikhomirov was also supported by Program of UD RAS, project No 15-16-1-3. A.~Naumov was supported by  RFBR N~16-31-00005, President's of Russian Federation Grant for young scientists N~4596.2016.1 and by Hong Kong RGC GRF~403513}

\keywords{Local semicircle law, rate of convergence to the semicircle law, rigidity of eigenvalues, delocalization of eigenvectors, moment conditions}

\date{\today}

\begin{abstract}
We consider a random symmetric matrix $\X = [X_{jk}]_{j,k=1}^n$ with  upper triangular entries being independent identically distributed random variables with mean zero and unit variance. We additionally suppose that
$\E |X_{11}|^{4 + \delta} =: \mu_{4+\delta} < C$ for some $\delta > 0$ and  some absolute constant $C$. Under these conditions we show that the typical  Kolmogorov distance  between the empirical spectral distribution function of eigenvalues of $n^{-1/2} \X$ and Wigner's semicircle law is of order $1/n$ up to some logarithmic correction factor. As a direct consequence of this result we establish that the semicircle law holds on a short scale. Furthermore,  we show for this finite moment ensemble  rigidity of eigenvalues and delocalization properties of the eigenvectors. Some numerical experiments are included illustrating  the influence of the tail behavior of the  matrix entries when only a small number of moments exist.
\end{abstract}

\maketitle

\section{Introduction and main result}
This paper is the second part of the project aimed to establish local semicircle law under moment conditions. For the readers convenience we shortly recall the most important notions
of our setup in the first part~\cite{GotzeNauTikh2015a} and give a very short survey of recent results. We consider a random symmetric matrix $\X = [X_{jk}]_{j,k=1}^n$ with upper triangular entries being independent random variables with mean zero and unit variance. Denote the $n$ eigenvalues of the symmetric matrix $\W: = \frac{1}{\sqrt n} \X$  in the increasing order by
$$
\lambda_1(\W) \le ... \le \lambda_n(\W)
$$
and introduce the eigenvalue counting function
$$
N_I(\W):= |\{1 \le k \le n: \lambda_k(\W) \in I\}|
$$
for any interval $I \subset \R$, where $|A|$ denotes the number of elements in the set $A$. Note that sometimes we shall omit $\W$ from the notation of $\lambda_j(\W)$.

It is well known since the pioneering work of E. Wigner~\cite{Wigner1958} that for any interval $I \subset \R$ of fixed length and independent of $n$
\begin{equation}\label{eq: wigner's semicircle law global regime}
\lim_{n \rightarrow \infty} \frac{1}{n}\E N_I(\W) = \int_I g_{sc}(\lambda) \, d\lambda,
\end{equation}
where
$$
g_{sc}(\lambda) := \frac{1}{2\pi} \sqrt{4 - \lambda^2} \one[|\lambda| \le 2]
$$
is the density function of Wigner's semicircle law.  Here and in what follows we denote by $\one[A]$ the indicator function of the set $A$. Wigner considered the special case when all $X_{jk}$ take only two values $\pm 1$ with equal probabilities.Wigner's semicircle law has been extended in various aspects, see, for example,~\cite{Arnold1967}, \cite{Pastur1973},~\cite{Girko1985uspexi}, \cite{GotTikh2006}, \cite{Naumov2013gaussiancase} and~\cite{GotNauTikh2012prob} and etc. For an extensive list of references we refer to the monographs~\cite{AndersonZeit},~\cite{BaiSilv2010} and~\cite{Tao2012}.

All these results hold for intervals $I$ of fixed length, independent of $n$, which typically contain a  macroscopically large number of eigenvalues, which means a number of order $n$.
It is of the great interest to investigate the case  of smaller intervals where the number of eigenvalues cease to be macroscopically large.
Here an appropriate analytical for asymptotic approximations is
the Stieltjes transform of the empirical spectral distribution function $F_n$, which is is given by
$$
m_n(z) := \int_{-\infty}^\infty \frac{d F_n(\lambda)}{\lambda - z}   = \frac{1}{n}\Tr (\W - z \I)^{-1} = \frac{1}{n} \sum_{j=1}^n \frac{1}{\lambda_j(\W) - z},
$$
where $z = u + i v, v \geq 0$. Taking the imaginary part of $m_n(z)$ we get
$$
\imag m_n(u + i v) = \int_{-\infty}^\infty \frac{v}{(\lambda - u)^2 + v^2} \, d F_n(\lambda) = \frac{1}{v}\int_{-\infty}^\infty K\left(\frac{u-\lambda}{v}\right) \, d F_n(\lambda)
$$
which is the kernel density estimator with Poisson's kernel $K$  and bandwidth $v$. For a meaningful estimator of the spectral density we cannot allow the distance $v$ to the real line, that is the bandwidth  of the kernel density estimator, to be smaller than the typical $\frac{1}{n}$ -distance between eigenvalues.
Hence, in what follows we shall be mostly interested in the situations when $v \gg \frac{1}{n}$.

Under rather general conditions for fixed $v > 0$ one may establish the convergence  of  $m_n(z)$ to the the Stieltjes transform of Wigner's semicircle law which is given by
$$
s(z) = \int_{-\infty}^\infty \frac{g_{sc}(\lambda)\, d\lambda}{\lambda-z} = -\frac{z}{2} + \sqrt{\frac{z^2}{4} - 1}.
$$

It is much more difficult  to establish the convergence in the region $1 \gg v \gg \frac{1}{n}$. Significant progress in that direction was recently made in a series of results by L.~Erd{\"o}s, B.~Schlein, H.-T.~Yau and et al., \cite{ErdosSchleinYau2009},~\cite{ErdosSchleinYau2009b} ,~\cite{ErdosSchleinYau2010},~\cite{ErdKnowYauYin2013}, showing
that with high probability uniformly in $u \in \R$
\begin{equation}\label{fluctuations of m_n around s}
|m_n(u+iv) - s(u+iv)| \le \frac{\log^\beta n}{nv},  \quad \beta > 0, 
\end{equation}
which they called  {\it local semicircle law}. It means  that the fluctuations of $m_n(z)$ around $s(z)$ are of order $(nv)^{-1}$ (up to a logarithmic factor). The value of $\beta$ may depend on $n$, to be exact $\beta: = \beta_n = c \log \log n$, where $c > 0$ denotes some constant. To prove~\eqref{fluctuations of m_n around s} in those papers ~\cite{ErdosSchleinYau2009}, \cite{ErdosSchleinYau2009b}, \cite{ErdosSchleinYau2010}
it was assumed that the distribution of $X_{jk}$ for all $1 \le j, k \le n$ has sub-exponential tails. Moreover in~\cite{ErdKnowYauYin2013} this assumption had been relaxed to requiring $\E |X_{jk}|^p \le \mu_p$ for all $p \geq 1$, where $\mu_p$ are some constants. Since there is meanwhile an extensive literature on the local semicircle law we refrain from  providing a complete list here and refer the  reader to the surveys of L.~Erd{\"o}s~\cite{Erdos2011} and T. Tao, V. Vu,~\cite{TaoVu2011surv}.

Our main goal in~\cite{GotzeNauTikh2015a} was to show that~\eqref{fluctuations of m_n around s} holds assuming that $\E |X_{jk}|^{4 + \delta} =: \mu_{4 + \delta} < \infty$. The first proof of a  result of this type follows from  a combination of arguments in a series of papers~\cite{ErdKnowYauYin2013a},~\cite{ErdKnowYauYin2012},~\cite{LeeYin2014}
(we sketched the underlying main ideas in the introduction of~\cite{GotzeNauTikh2015a}). In~\cite{GotzeNauTikh2015a} we gave a self-contained proof based on the method from~\cite{GotTikh2015},~\cite{GotzeTikh2014rateofconv} while at the same time reducing the power of $\log n$ from $\beta = c\log \log n$ to $\beta = 2$. Our work and some crucial bounds of our proof were motivated by the methods used in a recent paper of C.~Cacciapuoti, A.~Maltsev and B.~Schlein,~\cite{Schlein2014}, where the authors improved the log-factor dependence in~\eqref{fluctuations of m_n around s} in the sub-Gaussian case.

For a detailed statement of our result recall that the conditions $\Cond$ hold if $X_{jk}, 1\le j\le k\le n$ are i.i.d. with zero mean, unit variance and $\E |X_{11}|^{4+\delta} := \mu_{4+\delta} < \infty$ for some $\delta > 0$. We also introduce the following quantity
$$
\alpha = \frac{2}{4+\delta},
$$
which will control the level of truncation of the matrix entries. It was proved in the paper~\cite{GotzeNauTikh2015a}[Theorem~1.1] that under conditions $\Cond$  and any fixed $V > 0$ there exist positive constants $A_0, A_1$ and $C$ depending on $\alpha$ and $V$ such that
\begin{equation}\label{main result part 1 1}
\E |m_n(z) - s(z)|^p \le \left(\frac{Cp^2}{nv}\right)^p,
\end{equation}
for all $1 \le p \le A_1 (nv)^{\frac{1-2\alpha}{2}}$, $V \geq v \geq A_0 n^{-1}$ and $|u| \le 2+v$. Applying Markov's inequality we may rewrite this result in the following form
\begin{equation}\label{main result part 1 1 probability}
\Pb \left( |m_n(z) - s(z)| \geq \frac{K}{nv}\right) \le \left(\frac{Cp^2}{K}\right)^p,
\end{equation}
for all $1 \le p \le A_1 (nv)^{\frac{1-2\alpha}{2}}$, $V \geq v \geq A_0 n^{-1}$ and $|u| \le 2+v$. For application we are interested in the range of $v$, such that~\eqref{main result part 1 1} is valid for fixed $p$. It is clear that $V \geq v \geq C p^\frac{2}{1-2\alpha} n^{-1}$. Since we are interested in polynomial estimates we need to take $p$ of order $\log n$, which implies that $V \geq v \geq C n^{-1} \log^\frac{2}{1-2\alpha} n$. At the same time $K$ in~\eqref{main result part 1 1 probability} should be of order $\log^2 n$. Comparing with~\eqref{fluctuations of m_n around s} we get $\beta = 2$. If we would like to have better bound then any polinomial we should take $\beta = 3$.

In the region $|u| > 2 + v$ we may control only imaginary part. It was proved in~\cite{GotzeNauTikh2015a}[Theorem~1.1] that for any $u_0 > 0$ there exist positive constants $A_0, A_1$ and $C$ depending on $u_0, V$ and $\alpha$ such that
\begin{equation}\label{main result part 1 2}
\E |\imag m_n(z) - \imag s(z)|^p \le \left(\frac{Cp^2}{nv}\right)^p,
\end{equation}
for all $1 \le p \le A_1 (nv)^{\frac{1-2\alpha}{2}}$, $V \geq v \geq A_0 n^{-1}$ and $|u| \le u_0$.

In the current paper we apply \eqref{main result part 1 2} and  establish an estimate for the rate of convergence in probability of $F_n$ to $G_{sc}(x): = \int_{-\infty}^{x} g_{sc}(\lambda)\, d\lambda$, the rigidity of eigenvalues and delocalization properties of the eigenvectors. We will formulate these results in the sequel and discuss them.

Let us denote
$$
\Delta_n^{*}: = \sup_{x \in \R} |F_n(x) - G_{sc}(x)|.
$$
F. G{\"o}tze and A. Tikhomirov in~\cite{GotTikh2003} proved
that assuming $\E |X_{11}|^{12} = : \mu_{12} < \infty$, one may obtain the following estimate
$$
\E \Delta_n^{*} \le \mu_{12}^\frac16 n^{-\frac12}.
$$
Particularly this estimate implies by Markov's inequality that
\begin{equation}\label{eq: rate of convergence}
\Pb \left( \Delta_n^{*} \geq K \right ) \le \frac{\mu_{12}^\frac16}{ K n^{\frac12}}.
\end{equation}
It is easy to see from the previous bound that one may take $K \gg n^{-\frac{1}{2}}$. This result was extended by Bai  and et al., see~\cite{BaiHuPanZhou2011}, where it was shown that instead of existence of the $12$th moment it suffices  to assume existence of the $6$th moment. Applying~\eqref{main result part 1 1} we may obtain a much stronger bound.
\begin{theorem} \label{th: rate of convergence}
Assume that the condition $\Cond$ holds. Then there exist positive constants  $c$ and $C$ depending on $\alpha$ only such that for all $1 \le p \le c \log n$
$$
\Pb\left(\Delta_n^{*} \geq K \right) \le \frac{C^p \log^\frac{2p}{1-2\alpha} n}{K^p n^p}
$$
for all $K > 0$.
\end{theorem}
As a consequence we may choose $K \gg n^{-1}$ which is optimal. In particular, taking $K = n^{-1}\log^{\kappa} n$, where $\kappa: = 1+\frac{2}{1-2\alpha}$, we get that
\begin{equation}\label{eq: rate of convergence optimal}
\Pb\left(\Delta_n^{*} \geq \frac{\log^{\kappa} n}{n} \right) \le  \frac{1}{n^{c\log \log n}}.
\end{equation}
Under additional assumptions~\eqref{eq: rate of convergence optimal} was proved in~\cite{GotzeTikh2011rateofconv},~\cite{TaoVu2013} and~\cite{GotzeTikh2014rateofconv}. Comparing our result with~\cite{LeeYin2014}Theorem~3.6] note that we reduced the logarithmic factor and give explicit dependence on $\delta$. Using out technique  it is possible to reduce the power of logarithm in the stochastic size of $\Delta_n^{*}$  to $1$  assuming that the distribution of $X_{11}$ has  sub-Gaussian decay,  for details see Tikhomirov and Timushev (in preparation). The optimal power of logarithm is $\frac 12$ due to a  result of Gustavsson~\cite{Gustavsson2005}. In Section~\ref{numerical} we provide  some numerical experiments to illustrate  the bounds of Theorem~\ref{th: rate of convergence}.

Let $N[x - \frac{\xi}{2n}, x + \frac{\xi}{2n} ]: = N_I(\W)$ for $I =  [x - \frac{\xi}{2n}, x + \frac{\xi}{2n} ], \xi > 0$. The following result is the direct corollary of Theorem~\ref{th: rate of convergence}.
\begin{corollary}
Assume that  condition $\Cond$ holds. Then there exist positive constants  $c$ and $C$ depending on $\alpha$ such that for all $1 \le p \le c\log n$ and all $\xi > 0, K > 0$
$$
\Pb\left(\left | \frac{N[x - \frac{\xi}{2n}, x + \frac{\xi}{2n} ]}{\xi} - g_{sc}(x) \right | \geq \frac{K}{\xi} \right) \le \frac{C^p \log^\frac{2p}{1-2\alpha} n}{K^p n^p}.
$$
\end{corollary}

Another application of~\eqref{main result part 1 1}  is the following result which shows the rigidity of the eigenvalues. Let us define the  quantile  position of the $j$-th eigenvalue by
$$
\gamma_j: \quad \int_{-\infty}^{\gamma_j} g_{sc}(\lambda) \, d\lambda = \frac{j}{n}, \quad 1 \le j \le N.
$$
We will prove the following theorem.
\begin{theorem}\label{th: rigidity}
Assume that the conditions $\Cond$ hold. Then \\
\noindent (i). For all $j \in [K, n - K+1]$ there exist constants $c$ and $C, C_1$ depending on $\alpha$ such that such that for  all $1 \le p \le c \log n$ we have
$$
\Pb(|\lambda_j - \gamma_j| \geq C_1 K [\min(j, n- j+1)]^{-\frac13} n^{-\frac23} ) \le \frac{C^p \log^\frac{2p}{1-2\alpha} n}{K^p}.
$$
\noindent (ii). Assume that $\delta = 4$. For all $j \leq K$ or $j \geq n - K + 1$ there exist constants $c$ and $C, C_1$ such that for  $5 \le p \le c \log n$ and any $0<\phi < 2$
$$
\Pb(|\lambda_j - \gamma_j| \geq C_1 K [\min(j, n- j+1)]^{-\frac13} n^{-\frac23} ) \le \frac{C}{n^{2-\phi}} + \frac{C^p \log^{18p} n}{K^p}.
$$
\end{theorem}
Let us complement the results of this theorem by the following remarks. First we refer the interested reader to relevant results~\cite{Gustavsson2005} (Gaussian case),~\cite{ErdKnowYauYin2013}[Theorem~7.6], \cite{ErdKnowYauYin2013a}[Theorem~2.13], \cite{GotzeTikh2011rateofconv}[Remark~1.2], \cite{LeeYin2014}[Theorem~3.6] and~\cite{Schlein2014}[Theorem~4]. In particular, the result under comparable moment conditions  in~\cite{LeeYin2014}[Theorem~3.6] has an additional factor $\log^{c \log \log n} n$ which in our case may be reduced to $\log^{\kappa} n$.

The bound in the bulk of the limit spectrum, that is part (i), holds for all $\delta > 0$. Since the proof of this part is based on Theorem~\ref{th: rate of convergence} we expect that it should be valid for $\delta = 0$ as well. It is shown in the proof that with high probability $n - n \Delta_n^{*}$ eigenvalues lie in the support of the semicircle law. Applying this fact we may use the well-known  Smirnov transform from mathematical statistics together with the	
 bound from Theorem~\ref{th: rate of convergence}. Concerning the edges of the limit spectrum, that is part (ii), we have to assume  in addition that there exist a moment of order eight (corresponding to $\delta = 4$) to prove part $(i)$. In this step we use  ideas from~\cite{Schlein2014}[Lemma~8.1] and~\cite{ErdKnowYauYin2013}[Theorem~7.6]. It is still possible to get a bound for smaller $\delta$, $0 < \delta < 4$, but here our methods allow to prove that the estimate in part (ii) holds with  small probability of order $n^{-\varepsilon}$ only, where $\varepsilon : = \varepsilon(\delta) > 0$.
In order to improve this error to $O(n^{-2+\phi})$ we have to assume the existence of eight moments. The main problem here is to estimate the distance between $\max_{1 \le k \le n} |\lambda_k(\W)|$ and $\max_{1 \le k \le n} |\lambda_k(\hat \W)|$, where $\hat \W$ is the random matrix with entries from $\W$, but truncated on the level of order $n^\alpha$ (see the definition in the proof of Theorem~\ref{th: rigidity}).
This dependence on the tails of the distribution of entries  is illustrated in Section~\ref{numerical} with numerical experiments, where we try to explain the role of matrix truncation.

To prove Theorem~\ref{th: rigidity} we need to apply  stronger bounds for the distance between Stieltjes transforms then~\eqref{main result part 1 1}. Let us denote
\begin{equation}\label{eq: def gamma}
\gamma: = \gamma(u):= ||u| - 2|.
\end{equation}
We say that the conditions $\CondTwo$ hold if $\Cond$ are satisfied and $|X_{jk}| \le D n^\alpha, 1 \le j,k \le n$, where $D: =D(\alpha)$ is some positive constant. We also denote
$$
\mathcal E_p: = \frac{C^p p^p }{n^p (\gamma + v)^p} +  \frac{C^p p^{3p}}{(nv)^{2p} (\gamma + v)^\frac{p}{2}}  + \frac{C^p}{n^p v^\frac{p}{2} (\gamma + v)^\frac{p}{2}} + \frac{C^p p^p }{(nv)^\frac{3p}{2} (\gamma + v)^\frac{p}{4} }.
$$
It was shown in~\cite{GotzeNauTikh2015a}[Theorem~1.2] that
assuming the conditions $\CondTwo$ hold, $u_0 > 2$ and $V > 0$, there exist positive constants $A_0, A_1$ and $C$ depending on $u_0, V$ and $\alpha$ such that
\begin{equation}\label{main result part 1 3}
\E |\imag m_n(z) - \imag s(z)|^p \le \mathcal E_p,
\end{equation}
for all $1 \le p \le A_1 (nv)^{\frac{1-2\alpha}{2}}$, $V \geq v \geq A_0 n^{-1}$ and $2 \le |u| \le u_0$. See Theorem~1.2 in~\cite{GotzeNauTikh2015a} for details.

We conclude this paper by showing delocalization  of  eigenvectors. This question has been intensively studied in many papers, for example, in~\cite{ErdosSchleinYau2009}~\cite{GotzeTikh2011rateofconv},~\cite{ErdKnowYauYin2013a} and~\cite{ErdKnowYauYin2012}. Let us denote by $u_j := (u_{j 1}, ... , u_{j n})$ the eigenvectors of $\W$ corresponding to the eigenvalue $\lambda_j(\W)$ .
\begin{theorem}\label{th: delocalization}
Assume that  conditions $\Cond$ hold with  $\delta = 4$. Then there exist positive constants $C$ and $C_1$ such that
$$
\Pb \left(\max_{1 \le j, k \le n} |u_{jk}|^2 \geq \frac{C_1 \log^8 n}{n} \right) \le  \frac{C}{n}.
$$
\end{theorem}
Similarly as in Theorem~\ref{th: rigidity} we restrict ourselves here to the case $\delta = 4$ only.

We mention here that it is possible to extend the result for $0 < \delta < 4$ but reducing the power in the bound in probability from $1$ to some positive constant $\varepsilon$ depending on $\delta$ only. In the case $\delta = 4$  our methods yield  the following bound
$$
\Pb \left(\max_{1 \le j, k \le n} |u_{jk}|^2 \geq \frac{C_1 \log^{4+\varepsilon'} n}{n} \right) \le  \frac{C}{n^{c(\varepsilon')}},
$$
for any $\varepsilon > 0$ and some positive constant $c(\varepsilon')$ depending on $\varepsilon'$. We omit the details. See Section~\ref{numerical} for numerical experiments illustrating this remark.

We finally  remark that applying a moment matching technique as used in \cite{ErdKnowYauYin2012}[Inequality~7.12], and \cite{ErdKnowYauYin2013a}[Remark~2.18] one may prove the following bound assuming the  conditions of Theorem~\ref{th: delocalization}
$$
\Pb \left(\max_{1 \le j, k \le n} |u_{jk}|^2 \geq \frac{C_1 \log^{8+\varepsilon} n}{n} \right) \le  \frac{C}{n^{2-\varepsilon}},
$$
for any $\varepsilon > 0$.

\subsection{Notations} \label{sec: notation}
Throughout the paper we will use the following notations. We assume that all random variables are defined on common probability space $(\Omega, \mathcal F, \Pb)$ and denote by $\E$ the mathematical expectation with respect to $\Pb$.

We denote by $\R$ and $\C$ the set of all real and complex numbers. We also define $\C^{+}: = \{z \in \C: \imag  z \geq 0\}$. Let $\T = [1, ... , n]$ denotes the set of the first $n$ positive integers. For any $\J \subset \T$ introduce $\T_{\J}: = \T \setminus \J$.

For any matrix $\W$ together with its resolvent $\RR$ and Stieltjes transform $m_n$
we shall systematically use the corresponding notions	
 $\W^{(\J)}, \RR^{(\J)}, m_n^{(\J)}$, respectively,  for the  sub-matrix  of $\W$ with entries $X_{jk}, j,k \in \T\setminus \J$.

By $C$ and $c$ we denote some absolute positive constants.

For an arbitrary matrix $\A$ taking values in $\C^{n \times n}$ we define the operator norm by $\|\A\|: = \sup_{x \in \R^n: \|x\| = 1} \|\A x\|_2$, where $\|x\|_2 : = \sum_{j = 1}^n |x_j|^2$. We also define the Hilbert-Schmidt norm by $\|\A\|_2: = \Tr \A \A^{*} = \sum_{j,k = 1}^n |\A_{jk}|^2$.
By $\binom{2m}{m}$ we denote the binomial number $\frac{(2m)!}{m!m!}$.

\subsection{Acknowledgment} We would like to thank L.~Erd{\"o}s and H.-T.~Yau for drawing our attention to  relevant previous results and papers in connection with the results of this paper, in particular, ~\cite{ErdKnowYauYin2012},~\cite{ErdKnowYauYin2013},~\cite{ErdKnowYauYin2013a} and~\cite{LeeYin2014}.

\section{Rate of convergence in probability}
In this section we prove Theorem~\ref{th: rate of convergence}. We estimate the difference between $F_n$ and $G_{sc}$ in Kolmogorov's metric via the distance between corresponding Stieltjes transforms. For this purpose we formulate the following smoothing inequality from~\cite{GotTikh2003}[Corollary~2.3], which allows to relate   distribution functions to their Stieltjes transforms. For all $x \in [-2, 2]$ let us define $\gamma(x): = 2 -|x|$. Given $\frac{1}{2} > \varepsilon > 0$ we introduce the following intervals $\mathbb J_\varepsilon: = \{x \in [-2, 2]: \gamma(x) \geq \varepsilon\}$ and $\mathbb J_\varepsilon^{'} : =\mathbb J_{\varepsilon/2}$.
\begin{lemma} \label{l: bound for delta}
Let $v_0 > 0, a > 0$ and $\frac{1}{2} > \varepsilon > 0$ be positive numbers such that
$$
\frac{1}{\pi} \int_{|u|\le a} \frac{1}{u^2 + 1} \, du = \frac{3}{4} = : \beta,
$$
and
$$
2 v_0 a \le \varepsilon^{\frac32}.
$$
Then for any $V > 0$ and $v' := v'(x): =  v_0/\sqrt{\gamma(x)}, x \in \mathbb J_\varepsilon^{'}$, there exist positive constants $C_1$ and $C_2$ such that the following inequality holds
\begin{align*}
\Delta_n^{*} &\le \int_{-\infty}^\infty |m_n(u + i V) - s(u + i V)|\, du + C_1 v_0 + C_2 \varepsilon^\frac32 \\
&\qquad\qquad\qquad+\sup_{x \in \mathbb J_\varepsilon^{'}} \left | \int_{v'}^V (m_n(x + i v) - s(x + i v)) \, dv \right |.
\end{align*}
\end{lemma}
\begin{proof}
See~\cite{GotTikh2003}[Corollary~2.3] or~\cite{GotzeTikh2014rateofconv}[Proposition~2.1].
\end{proof}
It what follows we will need the following version of this lemma.
\begin{corollary}\label{cor: bound for delta}
Assuming the conditions of Lemma~\ref{l: bound for delta} we have
\begin{align}\label{eq: smoothing inequality}
\E^\frac{1}{p}[\Delta_n^{*}]^p &\le \int_{-\infty}^\infty \E^\frac{1}{p}|m_n(u + i V) - s(u + i V)|^p\, du + C_1 v_0 + C_2 \varepsilon^\frac32 \nonumber \\
&\qquad\qquad\qquad+\E^\frac{1}{p} \sup_{x \in \mathbb J_\varepsilon^{'}} \left | \int_{v'}^V (m_n(x + i v) - s(x + i v)) \, dv \right |^p.
\end{align}
\end{corollary}
\begin{proof}
The proof is the direct consequence of the previous lemma and we omit it. For details the interested reader is referred to~\cite{GotzeTikh2014rateofconv}[Corollary 2.1].
\end{proof}
\begin{proof}[Proof of Theorem~\ref{th: rate of convergence}] We proceed as in the proof of Theorem~1.1 in~\cite{GotzeTikh2014rateofconv}. We choose in Corollary~\ref{cor: bound for delta} the following values for the parameters $v_0, \varepsilon$ and $V$. Let us take $v_0: = A_0 n^{-1} \log^{\frac{2}{1-2\alpha}} n$, $\varepsilon: = (2 v_0 a)^\frac23$ and $V: = 4$.
We may partition $\mathbb J_\varepsilon^{'}$ into $k_n := n^4$ disjoint subintervals of equal length. Let us denote the endpoints of these intervals by $x_k, k = 0, ... , k_n$. We get $-2 + \varepsilon = x_0 < x_1 < ... < x_{k_n} = 2 - \varepsilon$. For simplicity we denote $\Lambda_n(u+iv): = m_n(u+iv) - s(u+iv)$ but we will not omit the argument.
We start to estimate the second integral in the r.h.s. of~\eqref{eq: smoothing inequality}. It is easy to see that
\begin{align}\label{eq: second integral 0}
&\sup_{x \in \mathbb J_\varepsilon^{'}} \left| \int_{v'}^V \Lambda_n(x + i v) \, dv \right | \le \max_{1 \le k \le k_n} \sup_{x_{k-1} \le x \le x_k} \left| \int_{v'}^V \Lambda_n(x + i v) \, dv \right|.
\end{align}
Applying the Newton-Leibniz formula we may write
\begin{align}\label{eq: second integral 1}
\sup_{x_{k-1} \le x \le x_k} \left| \int_{v'}^V \Lambda_n(x + i v) \, dv \right| &\le \left| \int_{v'}^V \Lambda_n(x_{k-1} + i v) \, dv \right| \nonumber\\
&+ \int_{x_{k-1}}^{x_k}  \int_{v'}^V |\Lambda_n'(x + i v)| \, dv \,dx.
\end{align}
It follows from Cauchy's integral formula that for all $z = x + i v$ with $v \geq v_0$  we have
\begin{equation}\label{eq: second integral 2}
|\Lambda_n'(x + i v)| \le \frac{C}{v^2} \le C n^2.
\end{equation}
We may conclude from~\eqref{eq: second integral 1} and~\eqref{eq: second integral 2} that
\begin{align*}
&\sup_{x_{k-1} \le x \le x_k} \left| \int_{v'}^V \Lambda_n(x + i v) \, dv \right| \le \left| \int_{v'}^V \Lambda_n(x_{k-1} + i v) \, dv \right| + \frac{C}{n}.
\end{align*}
Applying this inequality to~\eqref{eq: second integral 0} and taking the mathematical expectation we obtain
\begin{align}\label{eq: second integral 3}
\E\sup_{x \in \mathbb J_\varepsilon^{'}} \left| \int_{v'}^V \Lambda_n(x + i v) \, dv \right |^p &\le \E \max_{1 \le k \le k_n} \left| \int_{v'}^V \Lambda_n(x_{k-1} + i v) \, dv \right|^p + \frac{C^p}{n^p}\nonumber \\
&\le \sum_{k=1}^{k_n} \left|  \int_{v'}^V \E^\frac{1}{p}\big|\Lambda_n(x_{k-1} + i v)\big|^p \, dv \right|^p + \frac{C^p}{n^p}.
\end{align}
Since $x \in \mathbb J_\varepsilon^{'}$ it follows from~\eqref{main result part 1 1} that
\begin{align}\label{eq: bound for difference of ST}
\E|\Lambda_n(x + i v)|^p \le \left(\frac{Cp^2}{nv}\right)^p.
\end{align}
Choosing $p = A_1(n v_0)^{\frac{1-2\alpha}{2}} = c \log n$ we finally get from~\eqref{eq: second integral 3} and~\eqref{eq: bound for difference of ST} that
\begin{align}\label{eq: estimate 1}
&\E^\frac{1}{p} \sup_{x \in \mathbb J_\varepsilon^{'}} \left| \int_{v'}^V \Lambda_n(x + i v) \, dv \right |^p \le \frac{C k_n^\frac{1}{p} \log^3 n }{n}   + \frac{C}{n} \le \frac{ C \log^3 n}{n}.
\end{align}
It remains to estimate the first of the integrals in~\eqref{eq: smoothing inequality}. Let us suppose that we have already shown the following bound
\begin{align}\label{eq: estimate 2 assumption}
 \E^\frac{1}{p}|\Lambda_n(u + i V)|^p \le \frac{Cp |s(z)|^{\frac{p+1}{p}}}{n},
\end{align}
valid for all $z = u + i V, u \in \R$. Hence,
\begin{align}\label{eq: estimate 2}
&\int_{-\infty}^\infty \E^\frac{1}{p}|\Lambda_n(u + i V)|^p\, du \le \frac{C p}{n} \int_{-\infty}^\infty \int_{-\infty}^\infty \frac{d u \,d G_{sc}(x)}{((x-u)^2 + V^2)^\frac{p+1}{2p}} \le \frac{C\log^2 n}{n}.
\end{align}
Combining now~\eqref{eq: smoothing inequality},~\eqref{eq: estimate 1} and~\eqref{eq: estimate 2}  we get
$$
\E^\frac{1}{p}[\Delta_n^{*}]^p \le \frac{C\log^\frac{2}{1-2\alpha} n}{n}.
$$
Since $\E^\frac{1}{p}[\Delta_n^{*}]^p$ is non-decreasing function of $p$, the last inequality remains valid for all $1 \le p \le c \log n$. To finish the proof of Theorem~\ref{th: rate of convergence} it remains to apply Markov's inequality
$$
\Pb\left(\Delta_n^{*} \geq K \right) \le  \frac{\E [\Delta_n^{*}]^p }{K^p} \le \frac{C^p \log^\frac{2p}{1-2\alpha} n}{K^p n^p}.
$$
We conclude the proof by~\eqref{eq: estimate 2 assumption}. To derive this bound we will proceed in the same way as in the proof of Theorem~2.1 in~\cite{GotzeNauTikh2015a}. The main difference is that we don't need to estimate $\E |\RR_{jj}|^p$, but we have to establish~\eqref{eq: estimate 2 assumption} for all $u \in \R$.
Since means  repeating the arguments in the proof of Theorem~2.1 in~\cite{GotzeNauTikh2015a} we shall  omit many details and routine calculations here.

Firstly  is easy to show that one can assume that the entries of $\X$ satisfy the conditions $\CondTwo$. We omit the details.

We start  with a recursive representation for the diagonal entries $\RR_{jj} = (\W - z \I)^{-1}$ of the resolvent.  We may express $\RR_{jj}$ in the following way
\begin{equation}\label{eq: R_jj representation 0}
\RR_{jj} = \frac{1}{-z + \frac{X_{jj}}{\sqrt n} - \frac{1}{n}\sum_{l,k \in T_j} X_{jk} X_{jl} \RR_{kl}^{(j)}},
\end{equation}
where $\RR^{(j)}$ is defined in Section~\ref{sec: notation}. Let $\varepsilon : = \varepsilon_{1j} + \varepsilon_{2j}+\varepsilon_{3j}+\varepsilon_{4j}$, where
\begin{align*}
&\varepsilon_{1j} =  \frac{1}{\sqrt n}X_{jj}, \quad \varepsilon_{2j} = -\frac{1}{n}\sum_{l \ne k \in T_j} X_{jk} X_{jl} \RR_{kl}^{(j)}, \quad \varepsilon_{3j} = -\frac{1}{n}\sum_{k \in T_j} (X_{jk}^2 -1) \RR_{kk}^{(j)}, \\
&\varepsilon_{4j}= \frac{1}{n} (\Tr \RR - \Tr \RR^{(j)}).
\end{align*}
In these notations we may rewrite~\eqref{eq: R_jj representation 0} as follows
\begin{equation}\label{eq: R_jj representation}
\RR_{jj} = -\frac{1}{z + m_n(z)} + \frac{1}{z + m_n(z)} \varepsilon_j \RR_{jj}.
\end{equation}
Let us denote $b(z): = z + 2 s(z), \, b_n(z) = b(z) + \Lambda_n(z)$ and
\begin{equation}\label{definition of T}
T_n: = \frac{1}{n} \sum_{j=1}^n \varepsilon_j \RR_{jj}.
\end{equation}
Applying~\eqref{eq: R_jj representation}  we arrive at  the following representation for $\Lambda_n$ in terms of $T_n$ and $b_n$
\begin{equation}\label{definition of Lambda in new notations}
\Lambda_n = \frac{T_n}{z + m_n(z) + s(z)} = \frac{T_n}{b_n(z)}.
\end{equation}
Now we show that for $V = 4$ and all $u \in \R$ one may estimate the denominator in~\eqref{definition of Lambda in new notations}. It is easy to check that
\begin{equation}\label{V = 4 1}
|m_n(z)| \le \frac{1}{4} \le \frac{1}{2} |z + s(z)| \quad \text{ and } \quad |s(z) - m_n(z)| \le \frac{1}{2}.
\end{equation}
These inequalities imply
\begin{equation}\label{V = 4 2}
|b_n(z)| \geq \frac{1}{2} |z + s(z)| \quad \text{ and } \quad |z + m_n(z)| \geq \frac{1}{2} |s(z) + z|.
\end{equation}
Moreover, since $1 + z s(z) + s^2(z) = 0$ we get
\begin{equation}\label{V = 4 3}
\frac{1}{|b_n(z)|} \le 2|s(z)| \quad \text{ and } \quad |m_n(z)| \le |s(z)|(1 + 2|T_n|).
\end{equation}
We rewrite~\eqref{definition of Lambda in new notations} in the following way
$$
\Lambda_n = \frac{1}{n} \sum_{j=1}^n \frac{\varepsilon_{4j} \RR_{jj}}{b_n(z)} + \frac{1}{n} \sum_{\nu = 1}^3 \sum_{j=1}^n \frac{\varepsilon_{\nu j} \RR_{jj}}{b_n(z)}.
$$
Since
$$
\sum_{j=1}^n \varepsilon_{4j} \RR_{jj} = \frac{1}{n}\Tr \RR^2 = m_n'(z)
$$
we get that
$$
\Lambda_n = \frac{1}{n}\frac{m_n'(z)}{b_n(z)}  + \frac{1}{n} \sum_{\nu = 1}^3 \sum_{j=1}^n \frac{\varepsilon_{\nu j} \RR_{jj}}{b_n(z)} = \frac{1}{n}\frac{m_n'(z)}{b_n(z)} + \widehat \Lambda_n,
$$
where we denoted
$$
\widehat \Lambda_n:= \frac{1}{n} \sum_{\nu = 1}^3 \sum_{j=1}^n \frac{\varepsilon_{\nu j} \RR_{jj}}{b_n(z)}.
$$
Let us introduce the function $\varphi(z) = \overline{z} |z|^{p-2}$. Then
\begin{align*}
\E|\Lambda_n|^p = \E \Lambda_n \varphi(\Lambda_n)  = \frac{1}{n}\E \frac{m_n'(z)}{b_n(z)} \varphi(\Lambda_n) + \E \widehat \Lambda_n\varphi(\Lambda_n).
\end{align*}
Applying the result of Lemma~\ref{appendix lemma resolvent inequalities 1} we obtain a bound for the first term of the r.h.s. of the previous equation
\begin{align}\label{eq: T_n^p bound 1}
\frac{1}{n}\left|\E \frac{m_n'(z)}{b_n(z)} \varphi(\Lambda_n)\right| &\le \frac{C|s(z)|^2}{n} (1+\E^{\frac{1}{p}} |T_n|^p) \E^{\frac{p-1}{p}} |\Lambda_n|^p.
\end{align}
Cauchy-Schwartz inequality, Lemmas~\eqref{appendix lemma varepsilon_1}--~\eqref{appendix lemma varepsilon_3} and $\max_j |\RR_{jj}| \le V^{-1}$ together imply that for all $p \le c \log n$ that
$$
\E|T_n|^p \le \E \left( \frac{1}{n} \sum_{j=1}^n |\varepsilon_j|^2\right )^\frac{p}{2} \left( \frac{1}{n} \sum_{j=1}^n |\RR_{jj}|^2\right )^\frac{p}{2}  \le C.
$$
From this inequality and~\eqref{eq: T_n^p bound 1} it follows that
\begin{align}\label{eq: T_n^p first step}
\E|\Lambda_n|^p \le |\E \widehat \Lambda_n\varphi(\Lambda_n)| + \frac{|s(z)|^2}{n} \E^{\frac{p-1}{p}} |\Lambda_n|^p.
\end{align}
Now we consider the term $\E \widehat \Lambda_n \varphi(\Lambda_n)$. We split it into three parts with respect to $\varepsilon_{\nu j}, \nu = 1,2,3$ obtaining
$$
\E \widehat \Lambda_n \varphi(\Lambda_n) = \frac{1}{n} \sum_{\nu = 1}^3 \sum_{j=1}^n \E\frac{\varepsilon_{\nu j} \RR_{jj}}{b_n(z)}\varphi(\Lambda_n) = \mathcal A_1 + \mathcal A_2 + \mathcal A_3.
$$
We rewrite $\mathcal A_{\nu}$ as a sum of two terms as follows
\begin{align*}
&\mathcal A_{\nu 1}: = \frac{s(z)}{n}\E \sum_{j=1}^n \frac{\varepsilon_{\nu j}}{b_n(z)} \varphi(\Lambda_n), \\
&\mathcal A_{\nu 2}: = \frac{1}{n} \sum_{j=1}^n \E \frac{\varepsilon_{\nu j} [\RR_{jj} - s(z)]}{b_n(z)} \varphi(\Lambda_n).
\end{align*}
From H\"older's inequality and Lemma~\ref{appendix lemma sum of varepsilon_1} with $q = 1$ it follows that
\begin{align}\label{eq: A_11}
&|\mathcal A_{11}| \le |s(z)|^2 \E^{\frac{1}{p}} \left|\frac{1}{n} \sum_{j=1}^n \varepsilon_{1j} \right|^{p}  \E^{\frac{p-1}{p}}|\Lambda_n|^p \le
\frac{Cp |s(z)|^2}{n} \E^{\frac{p-1}{p}}|\Lambda_n|^p.
\end{align}
To estimate $\mathcal A_{2 1}$ and $\mathcal A_{3 1}$ let us introduce the following notation
$$
\widetilde \Lambda_n^{(j)} : = \E ( \Lambda_n \big| \mathfrak M^{(j)}),
$$
where $\mathfrak M^{(j)}: = \sigma\{X_{lk}, l, k \in \T_j\}$.
Since $\E (\varepsilon_{\nu j} \big| \mathfrak M^{(j)}) = 0$ for $\nu = 2, 3$ it is easy to see that $\mathcal A_{\nu 1} = \mathcal B_{\nu 1} + \mathcal B_{\nu 2}$, where
\begin{align*}
&\mathcal B_{\nu 1} = \frac{s(z)}{n} \sum_{j=1}^n \E\frac{\varepsilon_{\nu j}}{b_n^{(j)}(z)} [\varphi(\Lambda_n) - \varphi(\widetilde \Lambda_n^{(j)})], \\
&\mathcal B_{\nu 2} = \frac{s(z)}{n} \sum_{j=1}^n \E\frac{\varepsilon_{\nu j} \varepsilon_{4 j} }{b_n(z) b_n^{(j)}(z)} \varphi(\Lambda_n).
\end{align*}
Applying Lemma~\ref{appendix lemma varepsilon_4} and~\eqref{V = 4 2} one may show that
\begin{equation}\label{bound for B nu 2}
|\mathcal B_{\nu 2}| \le \frac{C |s(z)|^2}{n} \E^\frac{p-1}{p} |\Lambda|^p.
\end{equation}
From Newton-Leibniz formula (see Lemma~\ref{appendix Taylor formula} for details)  and the simple inequality $(x+y)^p \le e x^p + (p+1)^p y^p, x, y > 0, p \geq 1$ we get
$$
|\mathcal B_{\nu 1}| \le \mathcal C_{\nu 1} + \mathcal C_{\nu 2},
$$
where
\begin{align*}
&\mathcal C_{\nu 1}: =  \frac{e p |s(z)|^2}{n} \sum_{j=1}^n \E|\varepsilon_{\nu j}| |\Lambda_n - \widetilde \Lambda_n^{(j)} ||\widetilde \Lambda_n^{(j)} |^{p-2},\\
&\mathcal C_{\nu 2}: =  \frac{p^{p-2}|s(z)|^2}{n} \sum_{j=1}^n \E|\varepsilon_{\nu j}| |\Lambda_n - \widetilde \Lambda_n^{(j)}|^{p-1}.
\end{align*}
 Applying the Schur complement formula (see for details~\cite{GotzeTikh2014rateofconv}[Lemma~7.23] or~\cite{GotTikh2003}[Lemma~3.3]) we get
\begin{equation}\label{eq: distance between traces}
\Tr \RR - \Tr \RR^{(j)} = (1+\eta_j)\RR_{jj},
\end{equation}
where $\eta_j: = \eta_{0j} + \eta_{1j} + \eta_{2j}$ and
\begin{align*}
&\eta_{0j}: = \frac{1}{n} \sum_{k \in T_j} [(\RR^{(j)})^2]_{kk} = (m_n^{(j)}(z))', \quad \eta_{1j}: = \frac{1}{n} \sum_{k \neq l \in \T_j} X_{jl} X_{jk} [(\RR^{(j)})^2]_{kl}, \\
&\eta_{2j}: = \frac{1}{n} \sum_{k \in \T_j} [X_{jk}^2 - 1] [(\RR^{(j)})^2]_{kk}.
\end{align*}
It follows from~\eqref{eq: distance between traces} and $\Lambda_n - \widetilde \Lambda_n^{(j)} = \Lambda_n - \Lambda_n^{(j)} - \E(\Lambda_n - \Lambda_n^{(j)} \big| \mathfrak M^{(j)})$ that
\begin{equation*}
\Lambda_n - \widetilde \Lambda_n^{(j)} = \frac{1 + \eta_{j0}}{n} [\RR_{jj} - \E(\RR_{jj}\big|\mathfrak M^{(j)})] + \frac{\eta_{1j} +\eta_{2j}}{n}\RR_{jj}
- \frac{1}{n}\E((\eta_{j1} + \eta_{j2})\RR_{jj}\big|\mathfrak M^{(j)}).
\end{equation*}
Let us introduce additional notations. We define $\hat \varepsilon_j: = \varepsilon_{1j} + \varepsilon_{2j} + \varepsilon_{3j}$ and
$$
a_n^{(j)} : = \frac{1}{z + m_n^{(j)}(z)}.
$$
It is easy to check that
\begin{align*}
|\RR_{jj} - \E(\RR_{jj}\big|\mathfrak M^{(j)})| &\le |a_n^{(j)}|( |\hat \varepsilon_j \RR_{jj}| + \E(|\hat \varepsilon_j \RR_{jj}|\big|\mathfrak M^{(j)})) \le
C |s(z)| ( |\hat \varepsilon_j| + \E(|\hat \varepsilon_j|\big|\mathfrak M^{(j)})) \\
&\le C |s(z)| ( |\hat \varepsilon_j| + C n^{-\frac12}),
\end{align*}
and similarly,
\begin{equation*}
|\eta_{1j} +\eta_{2j}||\RR_{jj}| \le C|s(z)| |\eta_{1j} +\eta_{2j}|(1 + |\varepsilon_j|).
\end{equation*}
Applying these inequalities we estimate $|\Lambda_n - \widetilde \Lambda_n^{(j)}|$ as follows
\begin{equation}\label{eq: bound for difference of Lambdas}
|\Lambda_n - \widetilde \Lambda_n^{(j)}| \le \frac{C |s(z)|}{n} ( |\hat \varepsilon_j| + C n^\frac12) +  \frac{C |s(z)|}{n} |\eta_{1j} +\eta_{2j}|(1 + |\varepsilon_j|)
\end{equation}
Let us introduce the following quantity $\beta : = \frac{1}{2\alpha}, \beta > 1$. Denote by $\zeta$ an arbitrary random variable such that the expectation $\E |\zeta|^\frac{4\beta}{\beta-1}$ exists. Then
$$
\E(\varepsilon_{\nu j} |\Lambda_n - \widetilde \Lambda_n^{(j)}||\zeta|\big|\mathfrak M^{(j)}) \le
B_1 + ... + B_6,
$$
where
\begin{align*}
&B_1: = \frac{C |s(z)|}{n}\E( |\varepsilon_{\nu j} \hat \varepsilon_j \zeta|\big|\mathfrak M^{(j)}), \qquad\qquad\, B_2: = \frac{C |s(z)|}{n^\frac32}\E( |\varepsilon_{\nu j}  \zeta|\big|\mathfrak M^{(j)}), \\
&B_3: = \frac{C |s(z)|}{n} \E(|\varepsilon_{\nu j} \eta_{1j}| |\zeta| \big|\mathfrak M^{(j)}), \qquad\quad\,  B_4: = \frac{C |s(z)|}{n}\E(|\varepsilon_{\nu j} \eta_{2j}||\zeta| \big|\mathfrak M^{(j)}),\\
&B_5: = \frac{C |s(z)|}{n} \E(|\varepsilon_{\nu j} \eta_{1j}||\varepsilon_j| |\zeta| \big|\mathfrak M^{(j)}), \qquad B_6: = \frac{C |s(z)|}{n}\E(|\varepsilon_{\nu j} \eta_{2j}||\varepsilon_j||\zeta| \big|\mathfrak M^{(j)}).
\end{align*}
Applying Lemmas~\ref{appendix lemma varepsilon_1}--~\ref{appendix lemma eta_2} one may check that
$$
\max_{k = 1, ... 6} B_k \le \frac{C |s(z)|}{n^2}\E^\frac{\beta-1}{2\beta}( |\zeta|^\frac{2\beta}{\beta-1}\big|\mathfrak M^{(j)}).
$$
Hence,
\begin{align}\label{eq: representation for difference in general case}
\E(\varepsilon_{\nu j} |\Lambda_n - \widetilde \Lambda_n^{(j)}||\zeta|\big|\mathfrak M^{(j)}) & \le \frac{C |s(z)|}{n^2}\E^\frac{\beta-1}{4\beta}( |\zeta|^\frac{4\beta}{\beta-1}\big|\mathfrak M^{(j)}).
\end{align}
Taking $\zeta = 1$ in~\eqref{eq: representation for difference in general case} we get
\begin{align}\label{bound for C nu 1}
\mathcal C_{\nu 1} &=  \frac{e p}{n} \sum_{j=1}^n \E|\widetilde \Lambda_n^{(j)} |^{p-2} \E(|\varepsilon_{\nu j}| |\Lambda_n - \widetilde \Lambda_n^{(j)} |\big|\mathfrak M^{(j)}) \nonumber\\
&\qquad\qquad\le \frac{C p |s(z)|^3}{n^3} \sum_{j=1}^n \E|\widetilde \Lambda_n^{(j)} |^{p-2} \le \frac{C p |s(z)|^3}{n^2 } \E^\frac{p-2}{p}|\Lambda_n |^p.
\end{align}
Similarly, applying~\eqref{eq: representation for difference in general case} with $\zeta = |\Lambda_n - \widetilde \Lambda_n^{(j)}|^{p-2}$ we obtain
\begin{equation}\label{bound for C nu 2 0}
\mathcal C_{\nu 2} \le \frac{C p^{p-2} |s(z)|^3}{n^3} \sum_{j=1}^n \E^\frac{\beta-1}{4\beta} |\Lambda -  \widetilde \Lambda_n^{(j)}|^\frac{4\beta(p-2)}{\beta-1}.
\end{equation}
It is easy to check (see~\eqref{eq: bound for difference of Lambdas}) that for all $q \geq 1$
\begin{equation}\label{eq: bound for difference of Lambdas 2}
\E |\Lambda -  \widetilde \Lambda_n^{(j)}|^q \le \frac{C^q |s(z)|^q}{n^q}.
\end{equation}
This inequality and~\eqref{bound for C nu 2 0} together imply
\begin{equation}\label{bound for C nu 2}
\mathcal C_{\nu 2} \le \frac{C^p p^{p-2} |s(z)|^{p+1}}{n^p}.
\end{equation}
The estimates~\eqref{eq: A_11},\eqref{bound for B nu 2},~\eqref{bound for C nu 1} and~\eqref{bound for C nu 2} yield
\begin{equation}\label{bound for A nu 1 all terms}
\sum_{\nu = 1}^3 \mathcal A_{\nu 1} \le \frac{Cp |s(z)|^2}{n} \E^{\frac{p-1}{p}}|\Lambda_n|^p + \frac{C p |s(z)|^3}{n^2 } \E^\frac{p-2}{p}|\Lambda_n |^p + \frac{C^p p^{p-2} |s(z)|^{p+1}}{n^p}.
\end{equation}
It remains to estimate $\mathcal A_{\nu 2}, \nu = 1, 2, 3$. Recall that
$$
\mathcal A_{\nu 2}: = \frac{1}{n} \sum_{j=1}^n \E \frac{\varepsilon_{\nu j} [\RR_{jj}  -s(z)]}{b_n(z)} \varphi(\Lambda_n).
$$
From the representation  $\RR_{jj} - s(z) = s(z)(\Lambda_n - \varepsilon_j)\RR_{jj}$ it follows that
$$
\mathcal A_{\nu 2} = \frac{s(z)}{n} \sum_{j=1}^n \E \frac{\varepsilon_{\nu j} (\Lambda_n - \varepsilon_j) \RR_{jj}}{b_n(z)} \varphi(\Lambda_n).
$$
We may bound $\mathcal A_{\nu 2}, \nu = 1, 2, 3$, by the sum of two terms (up to some constant) $\mathcal N_{\nu, 1}$ and $\mathcal N_{\nu, 2}, \nu = 1,2,3$, where
\begin{align*}
&\mathcal N_{\nu 1}: =  \frac{e|s(z)|^2}{n} \sum_{j=1}^n \E |\varepsilon_{\nu j}| | \Lambda_n - \varepsilon_j| |\widetilde \Lambda_n^{(j)}|^{p-1}, \\
&\mathcal N_{\nu 2}: =  \frac{p^{p-1} |s(z)|^2}{n} \sum_{j=1}^n \E |\varepsilon_{\nu j}| |\Lambda_n - \varepsilon_j|  |\Lambda_n - \widetilde \Lambda_n^{(j)}|^{p-1}.
\end{align*}
Let us consider $\mathcal N_{\nu 1}$. Applying Lemmas~\ref{appendix lemma varepsilon_1}--~\ref{appendix lemma varepsilon_4} we obtain
$$
\mathcal N_{\nu 1} \le \frac{C|s(z)|^2}{n} \sum_{j=1}^n \E |\widetilde \Lambda_n^{(j)}|^{p-1}  \E^\frac{1}{2} (|\varepsilon_{\nu j}|^2 \big|\mathfrak M^{(j)} ) \E^\frac{1}{2} (|   \Lambda_n - \varepsilon_j  |^2 \big|\mathfrak M^{(j)} ) \le \frac{C|s(z)|^2}{n} \E^\frac{p-1}{p} |\Lambda_n|^p.
$$
Similarly, in view of~\eqref{eq: bound for difference of Lambdas 2} we conclude
$$
\mathcal N_{\nu 2} \le \frac{C^pp^{p-1} |s(z)|^{p+1}}{n^p}.
$$
Finally we get the following inequality for the sum of the $\mathcal A_{\nu 2}, \nu = 1, 2, 3$
\begin{align}\label{eq: bound for A alpha 2}
\sum_{\nu = 1}^{3} \mathcal A_{\nu 2} &\le \frac{C|s(z)|^2}{n} \E^\frac{p-1}{p} |\Lambda_n|^p + \frac{C^p p^{p-1} |s(z)|^{p+1}}{n^p}.
\end{align}
Combining\eqref{bound for A nu 1 all terms} and~\eqref{eq: bound for A alpha 2}  we get
\begin{align*}
\E |\Lambda_n|^p &\le \frac{Cp |s(z)|^2}{n} \E^{\frac{p-1}{p}}|\Lambda_n|^p + \frac{C p |s(z)|^3}{n^2 } \E^\frac{p-2}{p}|\Lambda_n |^p
  + \frac{C^pp^{p-1} |s(z)|^{p+1}}{n^p}
\end{align*}
Applying Lemma~\ref{appendix eq: inequality for x power p 2} we obtain the following estimate
\begin{align*}
\E |\Lambda_n|^p &\le \frac{C^p p^p |s(z)|^{p+1}}{n^p},
\end{align*}
which concludes the proof.
\end{proof}

%%%%%%%%%%%%%%%%%%%%%%%%%%%%%%%%%%%%%%%%%%%%%%%%%%%%%%%%%%%%%%%%%%%%%%%%%%%%%%%%%%%%%%%%%%%%%%%%%%%%%%%%%%%%%%%%%%%%%%%%%%%%%%%%%%%%%%%%%%%%%%%%%%%%%%%%%%%%%%%%%%%%%%%%%%%%%%%%%%%%
\section{Rigidity of  eigenvalues}
In this section we prove Theorem~\ref{th: rigidity}. We start with a lemma which shows that with high probability all eigenvalues lie in the interval $[-2 - K n^{-\frac23}, 2 + K n^\frac23]$ for some large $K > 0$. Here we shall use  methods similar to those in~\cite{Schlein2014}[Lemma~8.1] and~\cite{ErdKnowYauYin2013}[Theorem~7.6]
	adapting them to our setup.
\begin{lemma}\label{l: spectral norm of W}
Assume the conditions $\Cond$ hold with $\delta = 4$. Then exist positive constants $c, C$ such that for any $0 < \phi < 2$
\begin{equation}\label{eq: spectral norm of W from 2}
\Pb\left (\|\W\| \geq 2 + \frac{K}{n^\frac23} \right) \le \frac{(C p^{12})^p}{K^{p}} + \frac{C}{n^{2-\phi}}.
\end{equation}
for all $4 < p \le c n^{\frac{1}{24}}$ and $K > 0$.
\end{lemma}
\begin{remark}
We remark here that for case $0 < \delta < 4$ we are not  getting a reasonable bound yet. We can  only guarantee the existence of some $\varepsilon: = \varepsilon(\delta) > 0$ such that~\eqref{eq: spectral norm of W from 2} holds with probability less then $C n^{-\varepsilon}$ for some $C$ depending on $\delta$ only. The main problem here is to estimate the distance between $\max_{1 \le k \le n}|\lambda_k(\W)|$ and $\max_{1 \le k \le n}|\lambda_k(\W)|$. So far we  can  estimate the probability of the event $\W \neq \hat \W$ only which  holds with very small probability depending on the level of truncation and hence, on $\delta$. We omit the details, but refer instead   to Section~\ref{numerical} with numerical results illustrating
this behavior.
\end{remark}
\begin{proof}[Proof of Lemma~\ref{l: spectral norm of W}] Recall that $\lambda_1(\W) \le ... \le \lambda_n(\W)$ and
$$
\|\W\| = \max_{1 \le j \le n} |\lambda_j(\W)|
$$
Hence, it is enough to prove that
$$
\Pb\left (\lambda_1 \le -2 - \frac{K}{n^\frac23} \right) \le  \frac{(C p^3)^p}{K^{p}} + \frac{C}{n^{2-\phi}} \quad \text{ or } \quad \Pb\left (\lambda_n \geq 2 + \frac{K}{n^\frac23} \right) \le \frac{(C p^3)^p}{K^{p}} + \frac{C}{n^{2-\phi}}.
$$
Without loss of generality we consider only the bound for $\lambda_1$, since the same proof is valid for $\lambda_n$. Now we need to truncate the entries of $\X$. We introduce the usual notations. We take an arbitrary $0 < \phi' < \frac{1}{4}$. Let $\hat X_{jk}: = X_{jk} \one[|X_{jk}| \leq D n^{\frac12 - \phi'}]$, $\tilde X_{jk}: = X_{jk} \one[|X_{jk}| \geq D n^{\frac12 - \phi'}] - \E X_{jk} \one[|X_{jk}| \geq D n^{\frac12 - \phi'}]$ and finally
$\breve X_{jk}: = \tilde X_{jk} \sigma^{-1}$, where $\sigma^2: = \E |\tilde X_{11}|^2$.
By $\hat \X, \tilde \X$ and $\breve \X$
we denote the symmetric random matrices with entries $\hat X_{jk}, \tilde X_{jk}$ and $\breve X_{jk}$ respectively. In a similar way we denote the resolvent matrices and corresponding Stieltjes transforms.
In this case we have
$$
\Pb(\W \neq \hat \W) \le \frac{C}{n^{2-\phi}},
$$
where $\phi: = 8 \phi'$. It what follows we may assume $\W = \hat \W$. Let us fix some large positive constant, say $u_0$. Then by Lemma~\ref{l: bound of operator norm under condtwo} and~\ref{l: bound of operator norm under cond} in the Appendix we obtain that
there exist some positive constants $c, C$  and $C_1$ depending on  $u_0$ such that
$$
\Pb(\|\W\| \geq u_0) \le e^{-n^c \log u_0} + \frac{C}{n^{2-\phi}} \le \frac{C_1}{n^{2-\phi}}.
$$
We conclude that
\begin{align*}
\Pb\left (\lambda_1(\W) \le -2 - \frac{K}{n^\frac23} \right) &\le \Pb\left (\lambda_1(\W) \le -u_0 \right) + \Pb\left (-u_0 \le \lambda_1(\W) \le -2 - \frac{K}{n^\frac23} \right) \\
&\le \Pb\left (-u_0 \le \lambda_1(\W) \le -2 - \frac{K}{n^\frac23} \right) + \frac{C_1}{n^{2-\phi}}.
\end{align*}
In order to estimate the probability of $\lambda_1(\W)$ to lie in the interval  $-u_0$ to $-2 - K n^{-\frac23}$ let us divide this interval into sub intervals. We denote
$$
\kappa_j: = \frac{K+j}{n^\frac23} \quad \text{ and } \quad v_j: = \frac{(K+j)^\frac15}{n^\frac23}.
$$
Then we define the following intervals $I_j: = [-2 - \kappa_{j+1}, - 2 - \kappa_j]$ for $j = 0, ..., j_N$, where $N$ is the smallest integer such that $2 + \kappa_{j+1} \geq u_0$.
Denote $x_j: = - 2 - \kappa_j$. By a union bound we may write
$$
\Pb\left(-u_0 \le \lambda_1(\W) \le -2 -  \frac{K}{n^\frac23} \right) \le \sum_{j = 0}^N \Pb(\lambda_1( \W \in I_j).
$$
By definition the intervals $I_j$ are of  length $|I_j| \le n^{-\frac23}$ and the event $\lambda_1( \W) \in I_j$ involves $|\lambda_1( \W) - x_j| \le |I_j| \le v_j$. We may take $z_j: = x_j + i v_j$ and note the following fact. Suppose that $\lambda_1( \W) \in I_j$ then
\begin{equation}\label{eq: imag m_n lower bound}
\imag  m_n(z_j) = \frac{1}{n} \sum_{k = 1}^n \frac{v_k}{(\lambda_k( \W) - x_k)^2 + v_k^2 } \geq \frac{1}{2nv_j}.
\end{equation}
For the imaginary part of $s(z)$ and $|u| \geq 2$ we have the following bound
$$
c \frac{v}{|b(z)|} \le \imag s(z) \le C \frac{v}{|b(z)|},
$$
where $b(z): = z + 2 s(z)$. Moreover, $c_1 \sqrt{\gamma + v} \le |b(z)| \le C_1 \sqrt{\gamma + v}$, recalling that $\gamma: = \gamma(u): = ||u| - 2|$ (see the definition~\eqref{eq: def gamma}). Taking $z: = z_j$ we write
\begin{equation}\label{eq: imag s upper bound}
\imag s(z_j) \le \frac{C v_j}{\sqrt \kappa_j}.
\end{equation}
It is easy to see that
$$
\frac{C v_j}{\sqrt \kappa_j} \le \frac{C v_j^2 n}{n v_j \sqrt \kappa_j} \le \frac{C (K+j)^\frac25}{n v_j (K+j)^\frac12} \le \frac{1}{4 n v_j}
$$
for $K$ large enough. Hence, applying~\eqref{eq: imag m_n lower bound} and~\eqref{eq: imag s upper bound} we get
$$
\imag m_n(z_j) - \imag s(z_j) \geq \frac{1}{2 n v_j} - \frac{C v_j}{\sqrt \kappa_j} \geq \frac{1}{4 n v_j}.
$$
Applying the definition of $\kappa_j$ and $v_j$ we write
\begin{align*}
&\frac{1}{n v_j} \geq \frac{(K+j)^\frac45}{n \kappa_j} \geq \frac{(K+j)^\frac14}{n (\kappa_j + v_j)},\\
&\frac{1}{n v_j} \geq \frac{(K+j)^\frac{7}{10}}{(nv_j)^2 \sqrt{\kappa_j}} \geq \frac{(K+j)^\frac14}{(nv_j)^2 \sqrt{\kappa_j + v_j}},\\
&\frac{1}{n v_j} \geq \frac{(K+j)^\frac{2}{5}}{n \sqrt{v_j} \sqrt{\kappa_j}} \geq \frac{(K+j)^\frac14}{n \sqrt{v_j} \sqrt{\kappa_j + v_j}},\\
&\frac{1}{n v_j} \geq \frac{(K+j)^\frac{7}{20}}{(nv_j)^\frac32 \kappa_j^\frac14} \geq \frac{(K+j)^\frac14}{(nv_j)^\frac32 (\kappa_j + v_j)^\frac14}.
\end{align*}
Let us introduce the following quantity, which is the sum of four terms on the r.h.s. of the previous inequalities,
$$
\Psi_j: = \frac{1}{n (\kappa_j + v_j)} + \frac{1}{(nv_j)^2 \sqrt{\kappa_j + v_j}} + \frac{1}{n \sqrt{v_j} \sqrt{\kappa_j + v_j}} + \frac{1}{(nv_j)^\frac32 (\kappa_j + v_j)^\frac14}.
$$
Therefore,  if $\lambda_1( \W) \in I_j$ then
\begin{align}\label{eq: imag lambda upper bound}
\imag \Lambda_n(z_j) =  \imag  m_n(z_j) - \imag s(z_j) \geq C (K+j)^\frac14 \Psi_j.
\end{align}
Applying now Lemmas~\ref{appendix: lemma trunc 0},~~\ref{appendix: lemma trunc 2} in the Appendix,~\eqref{main result part 1 3} (see Theorem~2.2 in~\cite{GotzeNauTikh2015a}) and~\eqref{eq: imag lambda upper bound}  we get
\begin{align*}
\Pb\left(-u_0 \le \lambda_1( \W) \le -2 -  \frac{K}{n^\frac23} \right) &\le \sum_{j = 0}^N \Pb(\lambda_1( \W) \in I_j) \\
&\le \sum_{j = 0}^N \Pb( |\imag \Lambda (z_j)| \geq  C (K+j)^\frac14 \Psi_j)\\
&\le \sum_{j = 0}^N \frac{(C p^3)^p}{(K+j)^\frac{p}{4}} \le  \frac{(C p^3)^p}{K^{\frac{p}{4}-1}}.
\end{align*}
The last inequality concludes the proof of the lemma.
\end{proof}

\begin{proof}[Proof of Theorem~\ref{th: rigidity}] We first investigate the case (i) when $j \in [K, n - K +1]$.
Without loss of generality we may assume that in this case $\lambda_j \in [-2, 2]$  since otherwise
$$
\Pb(\lambda_j \le -2) \le \Pb\left(F_n(-2) \geq \frac{j}{n}\right) \le \Pb\left(\Delta_n^{*} \geq \frac{K}{n}\right) \le \frac{C^p \log^\frac{2p}{1-2\alpha} n}{K^p}
$$
and
$$
\Pb(\lambda_j \geq 2) \le \Pb\left(F_n(2) \leq \frac{j}{n}\right) \le \Pb\left(\Delta_n^{*} \geq \frac{K}{n}\right) \le \frac{C^p \log^\frac{2p}{1-2\alpha} n}{K^p},
$$
where $1 \le p \le c \log n$ and we applied the fact that $G_{sc}(-2) = 0$ and $G_{sc}(2) = 1$. It was proved in~\cite{GotzeTikh2011rateofconv} (see Section~9 in the Appendix),  that there exist constants $c_1$ and $c_2$ such that
\begin{align}\label{eq: quantiles of semicircle law 1}
&c_1 x^\frac23 \le 2+ G_{sc}^{-1}(x) \le c_2 x^\frac23 \quad \text{ for } x \in \left[0, \frac{1}{2}\right] \quad \text{ and } \\
\label{eq: quantiles of semicircle law 2}
&c_1 (1-x)^\frac23 \le 2- G_{sc}^{-1}(x) \le c_2 (1-x)^\frac23 \quad \text{ for } x \in \left[\frac{1}{2}, 1\right].
\end{align}
Obviously, the maximum in $\Delta_n^{*}$ is reached at the jump points of $F_n$, i.e.
$$
\Delta_n^{*} = \max_{1\le k \le n} |F_n(\lambda_k) - G_{sc}(\lambda_k)| = \max_{1 \le k \le n} \left| \frac{k}{n} - G_{sc}(\lambda_k)\right|.
$$
This fact implies that for every $j$ there exists $\theta, |\theta | \le 1$ such that
$$
\lambda_j = G_{sc}^{-1}\left(\frac{j}{n} + \theta \Delta_n^{*} \right).
$$
By Taylor's formula we get
\begin{equation}\label{eq: representation for lambda k}
\lambda_j = G_{sc}^{-1}\left(\frac{j}{n} \right) + \E_\tau \frac{2\pi \theta \Delta_n^{*}}{\sqrt{4 - \left(G_{sc}^{-1}\left(\frac{j}{n} + \theta \Delta_n^{*} \right) \right)^2}}.
\end{equation}
Again applying Theorem~1.1 we obtain that
$$
\Pb\left(\Delta_n^{*} \leq \frac{K}{2n} \right) \geq 1 - \frac{C^p \log^\frac{2p}{1-2\alpha} n}{K^p}.
$$
This means that without loss of generality we may assume that $\Delta_n^{*} \leq \frac{K}{2n}$. It remains to consider two cases. In the first, $ 2\Delta_n^{*} \le \frac{j}{n} \le \frac{1}{2} - \theta \Delta_n^{*}$ we may apply~\eqref{eq: quantiles of semicircle law 1} and conclude
$$
\sqrt{4 - \left(G_{sc}^{-1}\left(\frac{j}{n} + \theta \Delta_n^{*} \right) \right)^2} \geq c_1 \left| \frac{j}{n} + \theta \Delta_n^{*} \right |^\frac13 \geq c_1' \left( \frac{j}{n} \right)^\frac13,
$$
for some positive constant $c_1'$. This inequality together with~\eqref{eq: representation for lambda k} yield that
$$
|\lambda_j - \gamma_j| \le C_1 \Delta_n^{*} \left(\frac{n}{j}\right)^\frac13.
$$
In the opposite case we apply~\eqref{eq: quantiles of semicircle law 2} and obtain
$$
|\lambda_j - \gamma_j| \le C_1 \Delta_n^{*} \left(\frac{n}{n-j+1}\right)^\frac13.
$$
Combining the last two inequalities we get
$$
\Pb\left (|\lambda_j - \gamma_j| \le  C_1 K [\min(j, n- j+1)]^{-\frac13} n^{-\frac23} \right ) \geq 1 - \frac{C^p \log^\frac{2p}{1-2\alpha} n}{K^p}.
$$
\noindent (ii). We now turn our attention to the case $j \le K$ or $j \geq n - K+1$ and without loss of generality we restrict ourselves to the first one. Let us denote
$$
l: = \frac{C_1 K }{n^\frac23} \left(\frac{1}{j}\right)^\frac13.
$$
In the opposite case we take $l: = C_1 K (n-j+1)^{\frac13}n^{-\frac23}$. It is easy to see that
$$
\Pb(|\lambda_j - \gamma_j| \geq l) \le \Pb(|\lambda_j - \gamma_j| \geq l, \lambda_j > \gamma_j) + \Pb(|\lambda_j - \gamma_j| \geq l, \lambda_j < \gamma_j).
$$
The first case when $ \lambda_j > \gamma_j$ is trivial since in this situation $\lambda_j > j_1 \geq -2 + c_1 n^{-2/3}$ (see~\eqref{eq: quantiles of semicircle law 1}) and we may repeat the calculations above and get
$$
\Pb(|\lambda_j - \gamma_j| \geq l, \lambda_j > \gamma_j) \le \frac{C^p \log^{4p} n}{K^p},
$$
for $1 \le p \le c \log n$. It remains to bound $\Pb(|\lambda_j - \gamma_j| \geq l, \lambda_j < \gamma_j)$. Again applying~\eqref{eq: quantiles of semicircle law 1} we get $\gamma_j \le -2 + c_2\left(\frac{j}{n}\right)^\frac13 $. Hence, choosing an appropriate constant $C_1$ we obtain
\begin{align*}
\Pb(|\lambda_j - \gamma_j| \geq l, \lambda_j < \gamma_j) &= \Pb(\lambda_j \le \gamma_j - l, \lambda_j < \gamma_j)  \\
&\le \Pb\left(\lambda_1 \le -2 + c_2\left(\frac{j}{n}\right)^\frac23 - \frac{C_1 K }{n^\frac23} \left(\frac{1}{j}\right)^\frac13  \right)\\
&\le \Pb\left(\lambda_1 \le -2 -  \left( \frac{K}{n}   \right)^\frac{2}{3}  \right) \le \frac{(C p^{12})^p}{K^{\frac{2p}{3}}} + \frac{C}{n^{2-\phi}}.
\end{align*}
\end{proof}

%%%%%%%%%%%%%%%%%%%%%%%%%%%%%%%%%%%%%%%%%%%%%%%%%%%%%%%%%%%%%%%%%%%%%%%%%%%%%%%%%%%%%%%%%%%%%%%%%%%%%%%%%%%%%%%%%%%%%%%%%%%%%%%%%%%%%%%%%%%%%%%%%%%%%%%%%%%%%%%%%%%%%%%%%%%%%%%%%%%%%%%%%%%

\section{Delocalization of eigenvectors}
In this section we prove Theorem~\ref{th: delocalization}. Here  we shall apply the following result from~\cite{GotzeNauTikh2015a}[Lemma~4.1]. Let us denote
\begin{equation*}
\mathbb D: = \{z = u+iv \in \C: |u| \le u_0, V \geq v \geq v_0: = A_0 n^{-1} \},
\end{equation*}
where $u_0, V > 0 $ are any fixed real numbers and $A_0$ is some large constant to be determined below. Then assuming  the conditions $\CondTwo$ there exist a positive constant $C_0$ depending on $u_0, V$ and positive constants $A_0, A_1$ depending on $C_0, \alpha$
such that for all $z \in \mathbb D$ and $1 \le p \le A_1(nv)^{\frac{1-2\alpha}{2}}$ we have
\begin{equation}\label{eq: main lemma first statement}
\max_{1 \le j \le n} \E|\RR_{jj}(z)|^p \le C_0^p.
\end{equation}
\begin{proof}[Proof of Theorem~\ref{th: delocalization}] Let us introduce the following distribution function
$$
F_{nj}(x): = \sum_{k=1}^n |u_{jk}|^2 \one[\lambda_k(\W)\le x].
$$
Using the eigenvalue decomposition of $\W$ it is easy to see that
$$
\RR_{jj}(z) = \sum_{k=1}^n\frac{|u_{jk}|^2}{\lambda_k(\W) - z} =  \int_{-\infty}^\infty \frac{1}{x - z} \, d F_{nj}(x),
$$
which means that $\RR_{jj}(z)$ is the Stieltjes transform of $F_{nj}(x)$. For any $\lambda > 0$ we have
\begin{equation}\label{eq: Q nj}
\max_{1 \le k \le n} |u_{jk}|^2 \le \sup_x (F_{nj}(x + \lambda) - F_{nj}(x)) = : Q_{nj}(\lambda).
\end{equation}
Furthermore, it is easy to check that
\begin{equation}\label{eq: Q nj 2}
Q_{nj}(\lambda) \le 2 \sup_u \lambda \imag \RR_{jj}(u + i\lambda).
\end{equation}
Indeed,
\begin{align*}
\sup_u \lambda \imag \RR_{jj}(u + i\lambda) &=  \sup_u \sum_{k=1}^n \frac{\lambda^2 |u_{jk}|^2}{\lambda^2 + (\lambda_k - u)^2} \\
&\geq \sup_u \sum_{k=1}^n \frac{\lambda^2 |u_{jk}|^2}{\lambda^2 + (\lambda_k - u)^2} \one[u \le \lambda_k \le u + \lambda] = \frac{1}{2} Q_{nj}(\lambda).
\end{align*}
To finish the proof we need to show that with high probability the r.h.s. of~\eqref{eq: Q nj 2} is bounded by $n^{-1} \log^8 n$.
Let us recall the  following notations.  Let $\hat X_{jk}: = X_{jk} \one[|X_{jk}| \leq D n^\frac38]$, $\tilde X_{jk}: = X_{jk} \one[|X_{jk}| \geq D n^\frac38] - \E X_{jk} \one[|X_{jk}| \geq D n^\frac38]$ and finally
$\breve X_{jk}: = \tilde X_{jk} \sigma^{-1}$, where $\sigma^2: = \E |\tilde X_{11}|^2$.
Let $\hat \X, \tilde \X$ and $\breve \X$ denote symmetric random matrices  with entries $\hat X_{jk}, \tilde X_{jk}$ and $\breve X_{jk}$ respectively. In a similar way we denote the resolvent matrices by $\hat \RR, \tilde \RR$ and $\breve \RR$.
In this case we have
$$
\Pb(\W \neq \hat \W) \le \frac{C}{n}.
$$
 Let $u_0 > 0$ denote  a large constant, whose exact value will be chosen later. Applying Lemmas~\ref{l: bound of operator norm under cond} and~\ref{l: bound of operator norm under condtwo} in the Appendix it follows  that
\begin{equation}\label{eq: sup of RR jj 0}
\Pb(\|\W \| \geq u_0) \le \frac{C}{n}.
\end{equation}
It what follows we may assume that $\|\W\| \le u_0$ and $\W = \hat \W$. Then for $|u| \geq 2u_0$ and $v > 0$ we get
\begin{equation}\label{eq: sup of RR jj 1}
|\RR_{jj}(u + i v)| \le \int_{-u_0}^{u_0} \frac{1}{\sqrt{(x-u)^2 + v^2}} \, dF_{nj}(x) \le \frac{1}{u_0} \le C,
\end{equation}
where $C$ is some large positive constant which will be chosen later.
It remains to estimate $|\RR_{jj}(u + i v)|$ for all $-2u_0 \le u \le 2u_0$. For simplicity let us denote this interval by $\mathcal U_0$, i.e. $\mathcal U_0: = [-2u_0, 2u_0]$.
Be the triangular inequality we may write $|\RR_{jj}| = |\hat \RR_{jj}| \le |\tilde \RR_{jj}| + |\hat \RR_{jj} - \tilde \RR_{jj}|$. Applying the simple equation
$$
\hat \RR_{jj} - \tilde \RR_{jj} = [\hat \RR (\hat \W - \tilde \W) \tilde \RR]_{jj}
$$
we get
$$
|\hat \RR_{jj} - \tilde \RR_{jj}| \le \| \hat \W -\tilde \W\| \|\ee_j^T \hat \RR\|_2 \|\tilde \RR \ee_j\|_2,
$$
where $\ee_j$ is a unit column-vector with all entries zero except for an entry one at
the position $j$. Using Lemma~\ref{appendix lemma resolvent inequalities 1} in the Appendix we conclude that
$$
|\hat \RR_{jj}| \le  |\tilde \RR_{jj}| + \frac{1}{v}  \|\hat \W -\tilde \W\|  \sqrt{|\hat \RR_{jj}||\tilde \RR_{jj}|}.
$$
It is easy to see that
$$
\|\hat \W - \tilde \W\|_2^2 = \frac{1}{n} \sum_{j,k} [\E |X_{jk}| \one[|X_{jk}| \geq D n^\frac{3}{8}]]^2  \le \frac{C}{n^4},
$$
We may take $v = v_0: = C_1 n^{-1} \log^8 n$, with $C_1 \geq A_0$. Applying the simple inequality $2 |a b| \le a^2 + b^2$ we get
\begin{equation} \label{eq: bound for RR_jj via hat RR jj}
\sup_{u \in \mathcal U_0}|\RR_{jj}| \le  3\sup_{u \in \mathcal U_0}|\tilde \RR_{jj}|.
\end{equation}
It remains to estimate $\sup_{u \in \mathcal U_0} |\tilde \RR_{jj}(u + i v_0)|$. It is easy to see that
\begin{equation}\label{eq: tilde R jj representation deloc}
\tilde \RR(z) = (\tilde \W - z\I)^{-1} = \sigma^{-1} (\breve \W - z \sigma^{-1}\I)^{-1}  = \sigma^{-1}\breve \RR(\sigma^{-1}z).
\end{equation}
Applying the resolvent equality we get
\begin{equation}\label{eq: overline R jj representation deloc}
\breve \RR(z) - \breve \RR(\sigma^{-1}z) = (z - \sigma^{-1}z) \breve \RR(z) \breve \RR(\sigma^{-1}z).
\end{equation}
Combining~\eqref{eq: tilde R jj representation deloc} and~\eqref{eq: overline R jj representation deloc} we obtain
\begin{align*}
|\tilde \RR_{jj}(z) - \breve \RR_{jj}(z)| \le (\sigma^{-1} - 1) |\breve \RR_{jj}(\sigma^{-1} z)| + \frac{|z|(\sigma^{-1} - 1)}{v} \sqrt{|\breve \RR_{jj}(z)| |\breve \RR_{jj}(\sigma^{-1}z)|}.
\end{align*}
It is easy to check that $(\sigma^{-1} - 1) \le C n^{-\frac32}$ and $\max(|z \breve \RR_{jj}(z)|, |z \breve \RR_{jj}(\sigma^{-1}z)|) \le C$ for some constant $C$.
Similarly to the previous calculations we get that
\begin{equation} \label{eq: bound for tilde RR_jj via breve RR jj}
\sup_{u \in \mathcal U_0}|\tilde \RR_{jj}| \le  3\sup_{u \in \mathcal U_0}|\breve \RR_{jj}|.
\end{equation}
Note, that the matrix $\breve \W$ satisfies the conditions $\CondTwo$. Applying~\eqref{eq: main lemma first statement} with $p = c \log n$ we obtain
$$
\Pb(|\breve \RR_{jj}(u+ i v_0)| \geq C_0 e^\frac{5}{c}) \le \frac{\E |\breve \RR_{jj}(u+i v_0)|^p}{(C_0 e^\frac{5}{c})^p} \le \frac{1}{n^{5}}.
$$
We may partition interval $\mathcal U_0$ into $k_n : = n^4$ disjoint subintervals of equal length, i.e $-2u_0 = x_0 \le x_1 \le ... \le x_{k_n} = 2 u_0$. Then by the Newton-Leibniz formula
\begin{align*}
\sup_{u \in \mathcal U_0} |\breve \RR_{jj}(u + i v_0)| &\le \max_{1\le k \le k_n} \sup_{x_{k-1} \le x \le x_k}|\breve \RR_{jj}(x + i v_0)|\\
& \le  \max_{1\le k \le k_n} |\breve \RR_{jj}(x_{k-1} + i v_0)| + \max_{1\le k \le k_n} \int_{x_{k-1}}^{x_k} |\breve \RR_{jj}'(u+iv_0)| \, du.
\end{align*}
We may write
$$
\max_{1\le k \le k_n} \int_{x_{k-1}}^{x_k} |\breve \RR_{jj}'(u+iv_0)| \, du \le \frac{C}{n}.
$$
Thus we arrive at
\begin{align}\label{eq: bound for probability of breve RR jj}
&\Pb\left(\sup_{u \in \mathcal U_0} |\breve \RR_{jj}(u + iv_0)| \geq 2C_0 e^\frac{5}{c} \right) \nonumber\\
&\qquad\qquad\qquad\qquad\le \sum_{k=1}^{k_n} \Pb \left (|\breve \RR_{jj}(x_{k-1} + iv_0)| \geq C_0 e^\frac{5}{c}\right) \le \frac{C }{n}.
\end{align}
We choose now $\lambda: = v_0$. In view of~\eqref{eq: Q nj},~\eqref{eq: Q nj 2},~\eqref{eq: bound for RR_jj via hat RR jj},~\eqref{eq: bound for tilde RR_jj via breve RR jj} and~\eqref{eq: bound for probability of breve RR jj} we get that there exist $C$ and $C_1$ such that
$$
\Pb \left(\max_{1 \le j, k \le n} |u_{jk}|^2 \leq \frac{C_1 \log^8 n}{n} \right) \geq 1 - \frac{C}{n},
$$
which concludes the proof.
\end{proof}

\section{Numerical simulations}\label{numerical}
The aim of this section is to illustrate by numerical experiments some effects arising in cases where  a only small number of moments of matrix entries are finite. We restrict ourselves to those statistics which correspond to  the main results of the current paper.

We start by choosing an appropriate distribution for the matrix entries. To this end consider a random variable $\xi$ which has the following density and distribution function depending on a parameter $\mu$
$$
f_\mu(x) = \frac{\mu-1}{x^\mu} \one[x \geq 1] \quad \text{ and } \quad F_\mu(x) = \left(1 - \frac{1}{x^{\mu-1}}\right) \one[x \geq 1].
$$
This choice guarantees a non zero skewness    i.e. the moment of order three
that differs from  the standard Gaussian distribution.
To ensure existence of $m$ finite moments requires to choose $\mu > m + 1$. In what follows we shall take $\mu = m + 1.1$. Let $\xi_{jk}$ denote i.i.d copies of $\xi$. Then we
consider
$$
X_{jk}: = \frac{\xi_{jk} - \E \xi}{\sqrt {\Var \xi}}
$$
which are combined in the random matrix $\X : = [X_{jk}]_{j,k=1}^n$ with $\E X_{jk} = 0$ and $\E X_{jk}^2 = 1$. As usual we also introduce the truncated (also normalized) random matrix $\breve \X$.

In Figure~\ref{ris:image wigner 2} we plotted the normalized frequency histogram of the eigenvalues of $\W$ for different $\mu$ and $n = 2000$. We use the simplest procedure dividing the range $[\lambda_1(\W), \lambda_n(\W)]$ into $m$ intervals of equal size. In our case we take $m = 70$.  We know from~\cite{Pastur1973} that to guarantee convergence to Wigner's semicircle law it is enough to have  finite second moments only. It is visible that for $\mu = 3.1$ (this case corresponds to a finite second moment only) convergence is rather poor. But starting from $\mu = 4.1$ one observes a rather fast convergence. It is easy from the picture that the width of histogram's bars depends on number of finite moments, indicating the fact that with growing number of finite moments the number of eigenvalues outside of the suppport of the semicircle law becomes smaller.
\begin{figure}[ht]
\begin{center}
\includegraphics[scale=0.45]{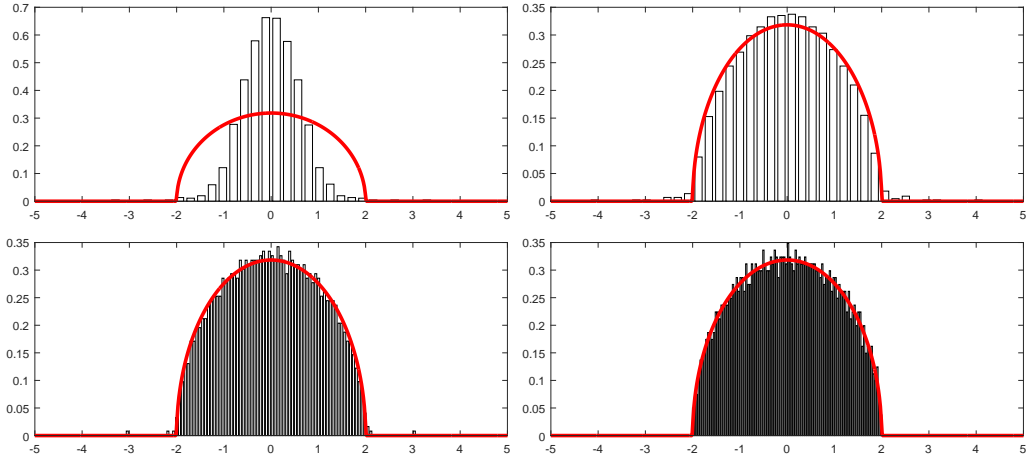}
\caption{Empirical spectral density of the eigenvalues of $\W$ for different $\mu$ and $n = 2000$.  In the top row $\mu = 3.1$ (on the left) and $\mu = 4.1$ (on the right). In the bottom row $\mu = 5.1$ (on the left) and $\mu = 9.1$ (on the right). Red line -- Wigner's semicircle law density function $g_{sc}$. }
\label{ris:image wigner 2}
\end{center}
\end{figure}

Let us consider the following statistics (motivated by the minimum error size, see \cite{Gustavsson2005})
$$
T_n: = \frac{n\Delta_n^{*}}{\sqrt{\log n}}.
$$
In Figure~\ref{ris:image KOLMOGOROV} we plotted $\E T_n$ (red line) with $\pm 1$  standard deviation around $\E T_n$ (black lines) for $n$ from $100$ to $5000$ with step $100$. We take the following values for $\mu$: 5.1 (top left), 7.1 (top right), 9.1 (bottom left) and Gaussian case (bottom right).
\begin{figure}[ht]
\begin{center}
\includegraphics[scale=0.45]{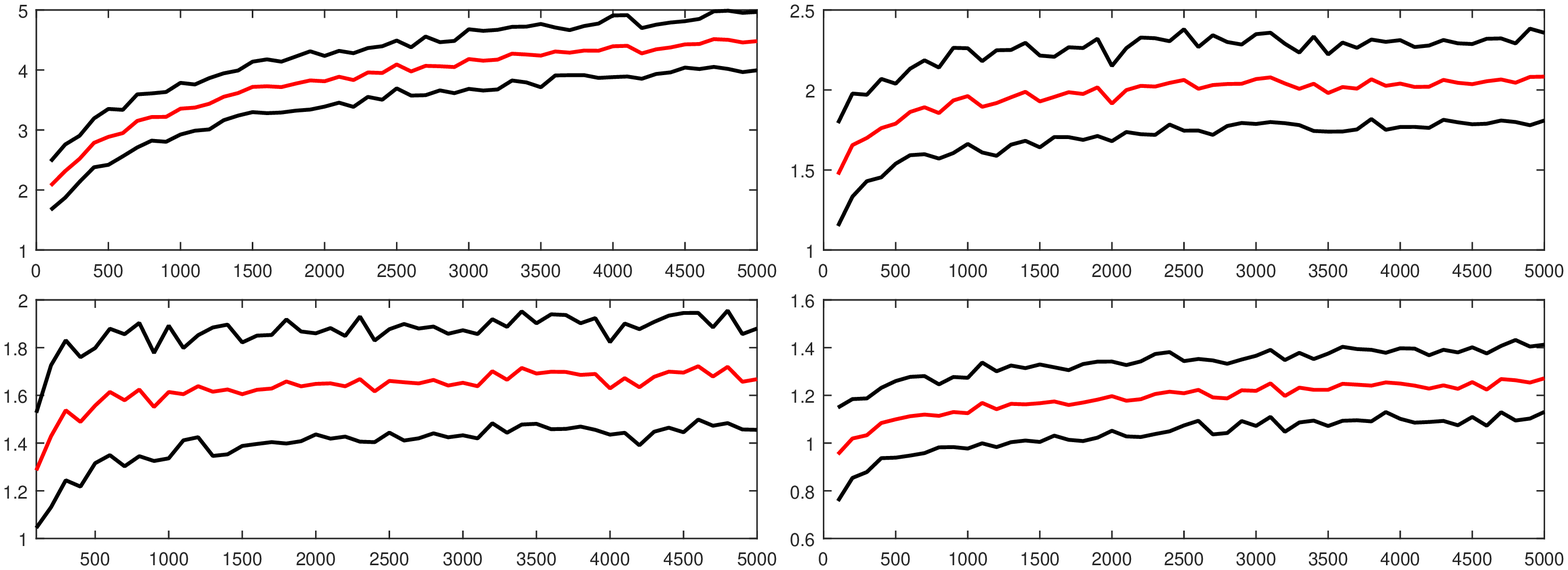}
\caption{ The plot of $\E T_n$  with $\pm 1$  standard deviation around the mean for $n$ from $100$ to $5000$ with step $100$. The values for $\mu$ here are $5.1; 7.1, 9.1$ and Gaussian distribution. }
\label{ris:image KOLMOGOROV}
\end{center}
\end{figure}

It is interesting to investigate the dependence of the largest eigenvalues on the tail behavior. For example, we consider $\lambda_n(\W)$ and study the following statistic
$$
\zeta_n = n^\frac23(\lambda_n(\W) - 2).
$$
In Figure~\ref{imag: all LE} we plotted on the left the distribution of $\zeta_n, n = 2000$ for the values $\mu = 5.1; 6.1; 7.1$ and $9.1$. On the right the distribution of truncated versions $\breve \zeta_n =  n^\frac23(\lambda_n(\breve \W) - 2)$ for the corresponding values of $\mu$. Here the red line is the Tracy--Widom density function with parameter $\beta = 1$, see~\cite{TrWid1994}. To plot the Tracy--Widom density function we applied the method of~\cite{Chiani2014} where the Tracy--Widom distribution has been
approximated by a gamma distribution with specific values. The impact of truncation is obvious from this graph. These figures motivate the remarks following Theorems~\ref{th: rigidity} and Lemma~\ref{l: spectral norm of W}.

\begin{figure}[ht]
\begin{center}
\includegraphics[scale=0.45]{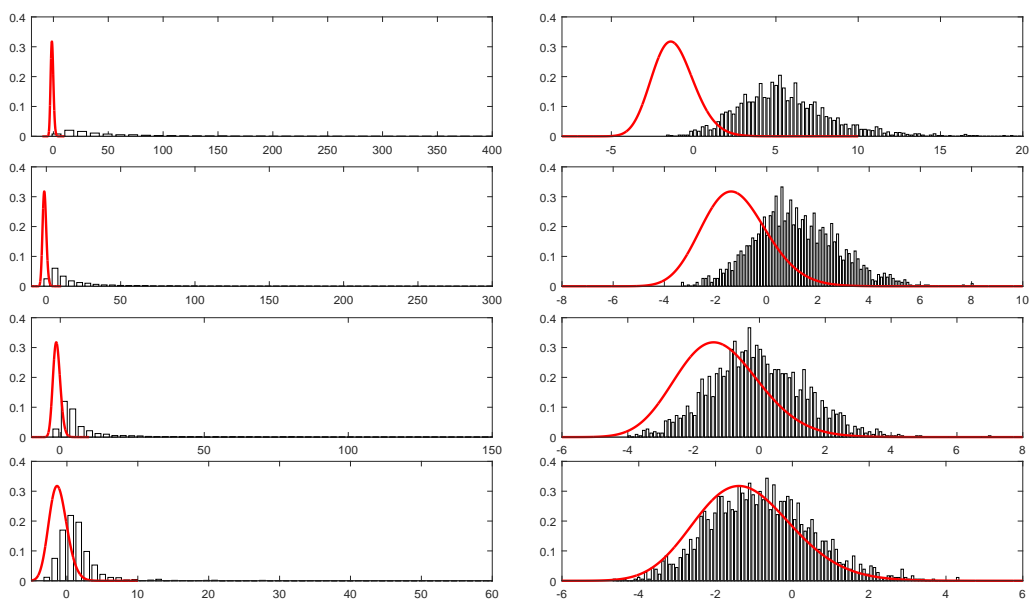}
\caption{On the left the distribution of $\zeta_n, n = 2000$ for the following values of $\mu = 5.1; 6.1; 7.1$ and $9.1$. On the right the distribution of $\breve \zeta_n$ for the corresponding values of $\mu$. Red line -- Tracy--Widom density function with parameter $\beta = 1$. }
\label{imag: all LE}
\end{center}
\end{figure}
Finally, we  consider simulations of  the empirical distribution of the following delocalization statistics
$$
V_n: = n \max_{1\le j, k \le n} |u_{jk}|^2,
$$
where $u_j := (u_{j 1}, ... , u_{j n})$ are the eigenvectors of $\W$ corresponding to the eigenvalue $\lambda_j(\W)$. In Figure~\ref{imag: vectors}  we plotted in the top row $V_n$ (on the left) and $\breve V_n$ (on the right), where $\breve V_n$ is $V_n$ with $\W$ replaced by $\breve \W$, for $\mu = 5.1$ and $n = 2000$.  The middle row shows the same statistics for $\mu = 9.1$. Finally, in the bottom row we compare $\breve V_n$ for $\mu = 9.1$ with $V_n$ in the Gaussian case.
\begin{figure}[ht]
\begin{center}
\includegraphics[scale=0.45]{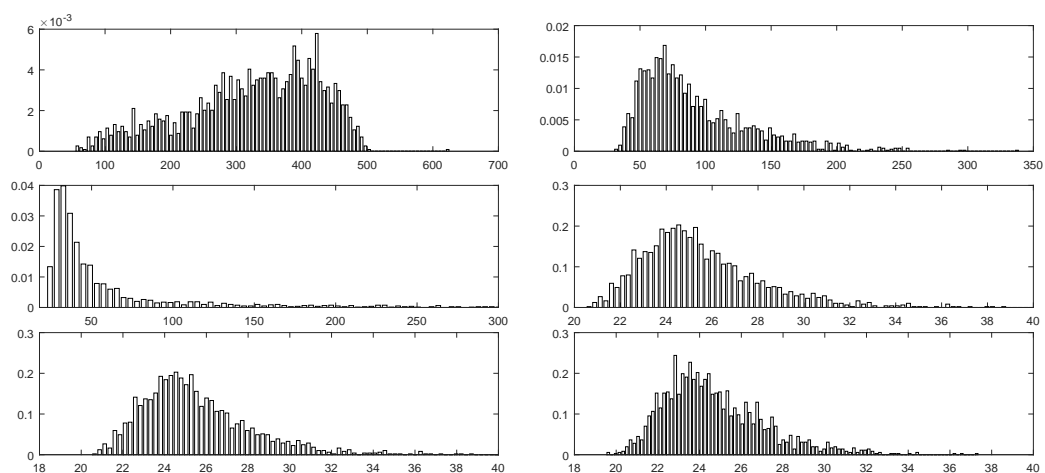}
\caption{In the top row $V_n$ (on the left) and $\breve V_n$ (on the right) for $\mu = 5.1$ and $n = 2000$. In the middle row the same statistics for $\mu = 9.1$. Finally, in the bottom row we compare $\breve V_n$ for $\mu = 9.1$ with $V_n$ in the Gaussian case }
\label{imag: vectors}
\end{center}
\end{figure}
It seems evident  that for the truncation $\breve V_n$ in the case $\mu = 9.1$ there is a good correspondence to Gaussian case.
Even in case of high moments, $\mu = 9.1$, The histogram of $V_n$  shows some deviation from the Gaussian case, which indicates a bad convergence rate.

\appendix

\section{Spectral norm of random matrices}

\begin{lemma}\label{l: bound of operator norm under condtwo}
Assume that the conditions $\CondTwo$ hold and let $K \geq 4$. Then there exists a constant  $c > 0$ depending on $\alpha$ such that
$$
\Pb(\|\W\| \geq K) \le e^{-n^c \log K}.
$$
\end{lemma}
\begin{proof}
It is common practice to control the extreme eigenvalues by the moment method, estimating $\E \Tr \X^k$ for large $k$ applying graph representation. The list of references is extensive, we only mention here some selected results. More details can be found in Chapter~2 of the monograph  of T.Tao~\cite{Tao2012}. In this paper we shall adopt a method due tp V. Vu from~\cite{Vu2007}.
Recall that
$$
\|\W\| = \max_{1 \le j \le n} |\lambda_j(\W)|
$$
and we obtain for even $k$
$$
\E \max_{1 \le j \le n} |\lambda_j(\W)|^k \le \sum_{j=1}^n \E\lambda_j(\W)^k = \E\Tr \W^k.
$$
In the following we shall use  notations and definitions used in~\cite{BaiSilv2010}.
A graph is a triple $(E, V, F)$, where $E$ is the set of edges, $V$ is the set of vertices, and $F$ is a function, $F: E \rightarrow V \times V$.
Let $\vect i = (i_1, ... , i_k)$ be a vector taking values in $\{1, ..., n\}^k$. For a vector $\vect i$ we define a $\Gamma$-graph as follows. Draw a horizontal line and plot the numbers $i_1, ... , i_k$ on it. Consider the distinct numbers as vertices, and draw $k$ edges $e_j$ from $i_j$ to $i_{j+1}, j = 1,...,k$, using $i_{k+1} = i_1$ by convention. Denote the number of distinct $i_j$'s by $t$. Such a graph is called a $\Gamma(k,t)$-graph.

Two $\Gamma(k,t)$-graphs are said to be {\it isomorphic if they can be converted into each other by a permutation of $(1, ..., n)$}. By this definition, all $\Gamma$-graphs are classified into isomorphism classes. We shall call the $\Gamma(k,t)$-graph canonical if it has the following properties:\\
1) Its vertex set is $\{1, .... , t\}$;\\
2) Its edge set is $\{e_1, ... , e_k \}$;\\
3) There is a function g from $\{1, ..., k\}$ onto $\{1, ... , t\}$ satisfying $g(1) = 1$ and $g(i) \le \max\{g(1), ... , g(i-1)\} + 1$ for $1 < i \le k$; \\
4) $F(e_i) = (g(i), g(i+1))$, for $i=1,...,k$, with the convention $g(k+1) = g(1) = 1$.

It is easy to see that each isomorphism class contains one and only one canonical $\Gamma$-graph that is associated with a function $g$, and a general graph in this class can be defined by $F(e_j) = (i_{g(j)}, i_{g(j+1)})$. Obviously, each isomorphism class contains $n(n-1)...(n-t+1)$ $\Gamma(k,t)$-graphs.

We expand the traces of powers of $\W$ in a sum
\begin{equation} \label{eq: moments of X}
\Tr \W^k = \frac{1}{n^\frac{k}{2}}\sum_{i_1, i_2, ... , i_k} X_{i_1 i_2} X_{i_2 i_3} ... X_{i_{k} i_1} = \frac{1}{n^\frac{k}{2}} \sum_{i_1, i_2, ... , i_k} X(\vect i),
\end{equation}
where the summation is taken over all sequences $\vect i = (i_1, ... , i_k) \in \{1, ... , n \}^k$. For each vector $\vect i$ we construct a graph $G(\vect i)$ as above and set  $X(G(\vect i)): = X(\vect i)$. Let us denote
\begin{equation} \label{eq: E n k t definition}
E(n,k,t) := \sum_{\Gamma(k,t)} \sum_{G(\vect i) \in \Gamma(k,t)} \E[X(G(\vect i))],
\end{equation}
where $\sum_{\Gamma(k,t)}$ is taken over all canonical $\Gamma(k,t)$-graphs with $t$ vertices and $k$ edges; and the summation $\sum_{G(\vect i) \in \Gamma(k,t)}$ is taken over all isomorphic graphs for a given canonical graph. It is easy to check that if $t \geq \frac{k}{2} + 1$ then $E(n,k,t) = 0$. Since $\E X_{i_1, i_2} = 0$ for all $1\le i_1 \le i_2 \le n$ and all $X_{i_1, i_2}$ are independent we may restrict ourself to the canonical graphs where each edge appears at least twice.

Let us also denote by $W(n,k, t)$ the number of these canonical graphs using $k$ edges and $t$ distinct vertices where each edge is used at least twice. It was proved in~\cite{Vu2007} that
\begin{equation}\label{eq: number of walks}
W(n,k,t) \le \binom{2t-2}{k} t^{2(k-2t+2)} 2^{2t - 2}.
\end{equation}
If a graph $G(\vect i)$ has $k$ edges and $t$ vertices then
\begin{equation}\label{eq: moments of one graph}
\E X(G(\vect i)) \le D^{k - 2(t-1)} n^{\alpha(k - 2(t-1))}.
\end{equation}
Thus applying~\eqref{eq: moments of X}--~\eqref{eq: moments of one graph} we obtain
\begin{align*}
\E \Tr \W^k &= \frac{1}{n^\frac{k}{2}} \sum_{t = 1}^{\frac{k}{2} + 1} E(n,k,t) \le \sum_{t = 1}^{\frac{k}{2} + 1} D^{k - 2(t-1)} n^{\alpha(k - 2(t-1))} n(n-1)...(n-t+1) W(n,k,t)\\
&\le \frac{1}{n^\frac{k}{2}} \sum_{t = 1}^{\frac{k}{2} + 1} D^{k - 2(t-1)} n^{\alpha(k - 2(t-1))} n(n-1)...(n-t+1) \binom{2t-2}{k} t^{2(k-2t+2)} 2^{2t - 2} \\
&\le \frac{1}{n^\frac{k}{2}} \sum_{t = 1}^{\frac{k}{2} + 1} S(n,k,t).
\end{align*}
It is easy to check that
$$
S(n,k,t-1) \le \frac{D^2 n^{2\alpha} k^6}{4 n} S(n,k,t).
$$
We may take $k = D^{-\frac{1}{3}} n^{\frac{1-2\alpha}{6}}$ and get $S(n,k,t-1) \le \frac12 S(n,k,t)$. It follows that
\begin{align}\label{eq: trace estimate}
\E \Tr \W^k  &\le \frac{1}{n^\frac{k}{2}}\sum_{t = 1}^{\frac{k}{2} + 1} S(n,k,t) \le  \frac{2}{n^\frac{k}{2}} S(n, k , k/2 +1) \nonumber \\
& = \frac{2 n(n-1)...(n-k/2)2^k}{n^\frac{k}{2}}  \le n 2^{k+1}.
\end{align}
Since $K \geq 4$, applying Markov's inequality for even $k$ and~\eqref{eq: trace estimate} we obtain
$$
\Pb(\|\W\| \geq K) \le \frac{\E \Tr \W^k}{K^k} \le 2 n \left( \frac{2}{K} \right)^{k} \le e^{-n^c \log K}.
$$
\end{proof}

We denote by $\hat X_{jk}: = X_{jk} \one[|X_{jk}| \geq Dn^\alpha]$, $\tilde X_{jk}: = \hat X_{jk} - \E \hat X_{jk}$ and finally $\breve X_{jk}: = \sigma^{-1} \tilde X_{jk}$, where $\sigma^2: = \E |\tilde X_{jk}|^2$.
By $\hat \W$, $\tilde \W$ and $\breve \W$
we denote the symmetric random matrices
with these entries.

\begin{lemma} \label{l: bound of operator norm under cond}
Under the conditions $\Cond$  for $K > 0 $ we have
$$
\Pb(\|\W\| \geq K) \le 2\Pb\left(\|\breve \W\| \geq \frac{K}{4} \right) + \Pb\left(\|\W - \hat \W \| \geq \frac{K}{4} \right).
$$
\end{lemma}
\begin{proof}
We start the proof  with the triangular inequality which yield the following estimate of $\|\W\|$
\begin{equation}\label{appendix eq: spectral norm zero step}
\|\W\| \le \|\W - \hat \W\| + \|\hat \W - \tilde \W\| + \|\tilde \W - \breve \W\| + \|\breve \W\|.
\end{equation}
It is easy to see that
\begin{equation}\label{appendix eq: spectral norm first step}
\|\hat \W - \tilde \W\|^2 \le \frac{1}{n} \sum_{j,k} [\E X_{jk}\one[|X_{jk}| \geq D n^\alpha]]^2 \le \frac{\mu_{4+\delta}^2}{D^{3+\delta} n^{2\alpha (3+\delta) - 1}} \le \frac{C}{n^2}.
\end{equation}
Since $\breve \W$ differs from $\tilde \W$ by a global change of variance we may write
$$
\|\tilde \W - \breve \W\| = (1 - \sigma) \|\breve \W\| \le (1 -\sigma^2) \|\breve \W\|.
$$
By definition of $\sigma$ we obtain
$$
(1 -\sigma^2) = \E |X_{jk}|^2 \one[|X_{jk}|\geq Dn^\alpha] \le \frac{\mu_{4+\delta}}{D^{2+\delta} n^{\alpha(2+\delta)}}.
$$
The last two inequalities together imply
\begin{equation}\label{appendix eq: spectral norm second step}
\|\tilde \W - \breve \W\| \le \frac{C}{n^{\alpha(2+\delta)}} \|\breve \W\|.
\end{equation}
Collecting the bounds~\eqref{appendix eq: spectral norm zero step}--~\eqref{appendix eq: spectral norm second step} we get the desired bound.
\end{proof}

\section{Truncation of matrix entries}

In this section we will show that the  conditions $\Cond$ allow to assume that  for all $1 \le j,k \le n$ we have $|X_{jk}| \le D n^{\alpha}$, where $D$ is some positive constant and
$$
\alpha = \frac{2}{4+\delta}.
$$
Let $\hat X_{jk}: = X_{jk} \one[|X_{jk}| \leq D n^\alpha]$, $\tilde X_{jk}: = X_{jk} \one[|X_{jk}| \geq D n^\alpha] - \E X_{jk} \one[|X_{jk}| \geq D n^\alpha]$ and finally
$\breve X_{jk}: = \tilde X_{jk} \sigma^{-1}$, where $\sigma^2: = \E |\tilde X_{11}|^2$.
Let again
 $\hat \X, \tilde \X$ and $\breve \X$ denote
 the symmetric random matrices with entries $\hat X_{jk}, \tilde X_{jk}$ and $\breve X_{jk}$ respectively. In a similar way we denote the corresponding  empirical spectral distribution functions, resolvent matrices and corresponding Stieltjes transforms.

\begin{lemma}\label{appendix: lemma trunc 1}
Under conditions $\Cond$ we have
$$
\E^\frac{1}{p} \sup_{x \in \R} |F_n(x) - \hat F(x)|^p \le \frac{C p}{n}.
$$
Moreover,
$$
\E|m_n(z) - \hat m_n(z)|^p \le \left(\frac{Cp}{nv}\right)^p.
$$
\end{lemma}
\begin{proof}
See in~\cite{GotzeNauTikh2015a}[Lemma~D.1].
\end{proof}

\begin{lemma}\label{appendix: lemma trunc 0}
Under conditions $\Cond$ we have
$$
\E|\tilde m_n(z) - \breve m_n(z)|^p \le \frac{C^p p^p \imag^p s(z)}{(nv)^p} + \frac{C^p p^{3p}}{(nv)^{2p}}.
$$
\end{lemma}
\begin{proof}
See in~\cite{GotzeNauTikh2015a}[Lemma~D.2].
\end{proof}

\begin{lemma}\label{appendix: lemma trunc 2}
Under conditions $\Cond$ we have
$$
\E|\tilde m_n(z) - \hat m_n(z)|^p \le \left(\frac{C}{nv}\right)^\frac{3p}{2}.
$$
\end{lemma}
\begin{proof}
See in~\cite{GotzeNauTikh2015a}[Lemma~D.3].
\end{proof}

\section{Auxiliary Lemmas }

We start this section with several lemmas providing inequalities for moments of linear and quadratic forms. Recall that
\begin{align*}
&\varepsilon_{1j} =  \frac{1}{\sqrt n}X_{jj}, \quad \varepsilon_{2j} = -\frac{1}{n}\sum_{l \ne k \in T_j} X_{jk} X_{jl} \RR_{kl}^{(j)}, \quad \varepsilon_{3j} = -\frac{1}{n}\sum_{k \in T_j} (X_{jk}^2 -1) \RR_{kk}^{(j)}, \\
&\varepsilon_{4j}= \frac{1}{n} (\Tr \RR - \Tr \RR^{(j)}).
\end{align*}
The following result is obvious but will be needed in the proof of Theorem~\ref{th: rate of convergence}
\begin{lemma}\label{appendix lemma varepsilon_1}
Under conditions $\CondTwo$ for $p \geq 1$ we have
$$
\E|\varepsilon_{1j}|^{2p} \le \frac{\mu_4 D^{2p-4}}{n^{p(1-2\alpha)+4\alpha}}.
$$
\end{lemma}
\begin{proof}
The proof follows directly from the definition of $\varepsilon_{1j}: = \frac{1}{\sqrt n} X_{jj}$.
\end{proof}

\begin{lemma}\label{appendix lemma sum of varepsilon_1}
Under conditions $\CondTwo$ for $p \geq 1$ and $q = 1,2$ there exists a positive constant $C$ depending on $\alpha$ such that
\begin{equation}\label{appendix eq sum of varepsilon_1 q = 1,2}
\E \left| \frac{1}{n} \sum_{j=1}^n \varepsilon_{1j}^q \right|^p \le  \frac{(C p)^p}{n^p}.
\end{equation}
\end{lemma}
\begin{proof}
See in~\cite{GotzeNauTikh2015a}[Lemma~A.4].
\end{proof}

The following Lemmas~\ref{appendix lemma varepsilon_2}--~\ref{appendix lemma eta_3} were proved in~\cite{GotzeNauTikh2015a}.
For completeness we state  them here again but for the special case of $v$ being a fixed constant denoted by $V$. In this case all inequalities obviously hold.
\begin{lemma}\label{appendix lemma varepsilon_2}
Under conditions $\CondTwo$ for $p \geq 2$ and $z = u + i V$ with some fixed $V > 0$ we have
\begin{align*}
\E |\varepsilon_{2j}|^p &\le C^{p}   \left(\frac{p^{\frac{3p}{2}}}{n^\frac{p}{2}} +\frac{p^{2p}}{n^{p(1-2\alpha)}} \right),
\end{align*}
where $C$ depends on $V$ and $\alpha$.
\end{lemma}
\begin{proof}
See ~\cite{GotzeNauTikh2015a}[Lemma~A.5].
\end{proof}

For $p = 2$ and $4$ we may give the better bound for quadratic form $\varepsilon_{2j}$. Let $\mathfrak M^{(j)}: = \sigma\{X_{lk}, l, k \in \T_j\}$.
\begin{lemma}\label{appendix lemma varepsilon_2 4 moment}
Under conditions $\Cond$ for $q = 2$ and $4$ we have
$$
\E(|\varepsilon_{2j}|^q \big|\mathfrak M^{(j)}) \le \frac{C}{n^{\frac{q}{2}}} \imag^{\frac{q}{2}} m_n^{(j)}(z),
$$
where $z = u + i V$ with some fixed $V>0$ and $C$ depending on $V$.
\end{lemma}
\begin{proof}
See ~\cite{GotzeNauTikh2015a}[Lemma~A.6].
\end{proof}

\begin{lemma}\label{appendix lemma varepsilon_3}
Under conditions $\Cond$ for $p \geq 2$ and $z = u + i V$ with some fixed $V>0$ we have
$$
\E |\varepsilon_{3j}|^p \le C^p\left (\frac{p^{\frac{p}{2}}}{n^\frac{p}{2}} +  \frac{p^p}{n^{p(1-2\alpha)}} \right ),
$$
where $C$ depends on $V$ and $\alpha$.
\end{lemma}
\begin{proof}
See~\cite{GotzeNauTikh2015a}[Lemma~A.7].
\end{proof}

For small $p \le \frac{1}{\alpha}$ we may write a better bound for $\varepsilon_{3j}$.
\begin{lemma}\label{appendix lemma varepsilon_3 small p}
Under conditions $\CondTwo$ for $2 \le p \leq \frac{1}{\alpha}$ and $z = u + iV$ with fixed $V>0$ we have
$$
\E(|\varepsilon_{3j}|^p \big|\mathfrak M^{(j)}) \le \frac{C}{n^{\frac{p}{2}}},
$$
where $C$ depends on $V$ and $\alpha$
\end{lemma}
\begin{proof}
See~\cite{GotzeNauTikh2015a}[Lemma~A.8].
\end{proof}

\begin{lemma}\label{appendix lemma varepsilon_4}
For $p \geq 2$ and $z = u + i V$ with fixed $V>0$ we have
$$
\E |\varepsilon_{4j}|^p \le \frac{1}{n^p}.
$$
\end{lemma}
\begin{proof}
See~\cite{GotzeNauTikh2015a}[Lemma~A.9].
\end{proof}

Recall the definition of $\eta_{\nu j}, \nu = 0, 1, 2$
\begin{align*}
&\eta_{0j}: = \frac{1}{n} \sum_{k \in T_j} [(\RR^{(j)})^2]_{kk} , \qquad \eta_{1j}: = \frac{1}{n} \sum_{k \neq l \in \T_j} X_{jl} X_{jk} [(\RR^{(j)})^2]_{kl},\\
&\eta_{2j}: = \frac{1}{n} \sum_{k \in \T_j} [X_{jk}^2 - 1] [(\RR^{(j)})^2]_{kk}.
\end{align*}

\begin{lemma}\label{appendix lemma eta_1}
Under conditions $\Cond$ for $2 \le p \le 4$ and $z = u + i V$ with fixed $V>0$ we have
$$
\E (|\eta_{1j}|^p \big|\mathfrak M^{(j)})\le  \frac{C}{n^\frac{p}{2}},
$$
where $C$ depends on $V$.
\end{lemma}
\begin{proof}
See~\cite{GotzeNauTikh2015a}[Lemma~A.10].
\end{proof}

\begin{lemma}\label{appendix lemma eta_2}
Under conditions $\CondTwo$ for $2 \le p \leq \frac{1}{\alpha}$ and $z = u + i V$ with fixed $V>0$ we have
$$
\E (|\eta_{2j}|^p \big|\mathfrak M^{(j)}) \le \frac{C}{n^{\frac{p}{2}}}.
$$
where $C$ depends on $V$ and $\alpha$.
\end{lemma}
\begin{proof}
See~\cite{GotzeNauTikh2015a}[Lemma~A.11].
\end{proof}

\begin{lemma}\label{appendix lemma eta_3}
For $p \geq 2$ and $z = u + i V$ with fixed $V>0$ we have
$$
|\eta_{0j}|^p  \le \frac{ \imag^p m_n^{(j)}(z)}{n^p }.
$$
\end{lemma}
\begin{proof}
See~\cite{GotzeNauTikh2015a}[Lemma~A.12].
\end{proof}

The following lemma provides estimates for  norms of vectors and matrices in terms of the resolvent.
\begin{lemma}\label{appendix lemma resolvent inequalities 1}
For any $z = u + i v \in \C^{+}$ we have
\begin{equation}\label{appendix lemma resolvent inequality 1}
\frac{1}{n} \sum_{l,k \in \T_{j}} |\RR_{kl}^{(j)}|^2 \le \frac{1}{v} \imag m_n^{(j)}(z).
\end{equation}
For any $l \in \T_{j}$
\begin{equation}\label{appendix lemma resolvent inequality 2}
\sum_{k \in \T_{j}} |\RR_{kl}^{(j)}|^2 \le \frac{1}{v} \imag \RR_{ll}^{(j)}.
\end{equation}
Moreover,
\begin{equation}
\label{appendix lemma resolvent square inequality 1}
\frac{1}{n} \big |\Tr (\RR^{(j)})^2 \big | \le \frac{1}{v} \imag m_n^{(j)}(z).
\end{equation}
\end{lemma}
\begin{proof}
See~\cite{GotzeNauTikh2015a}[Lemma~C.4, Lemma~C.5].
\end{proof}

Recall that $\varphi(z) = \overline z |z|^{p-1}$. In the following lemma we show how to estimate the difference between  $\varphi(\Lambda_n)$ and $\varphi(\widetilde \Lambda_n^{(j)})$.
\begin{lemma}\label{appendix Taylor formula}
	For $p \geq 2$ and arbitrary $j \in \T$ we have
	$$
	|\varphi(\Lambda_n)-\varphi(\widetilde \Lambda_n^{(j)})| \le p \E_{\tau}|\Lambda_n - \widetilde \Lambda_n^{(j)}||\widetilde \Lambda_n^{(j)} + \tau (\Lambda_n -\widetilde \Lambda_n^{(j)})|^{p-2},
	$$
	where $\E_{\tau}$ denotes expectation with respect to a random variable $\tau$ which is uniformly distributed on $[0, 1]$.
\end{lemma}
\begin{proof}
	The proof follows from the Newton-Leibniz formula applied to
	$$
	\hat \varphi(x) = \varphi( \widetilde \Lambda_n^{(j)} + x (\Lambda_n -\widetilde \Lambda_n^{(j)})), \quad x \in [0,1],
	$$
	and
	$$
	|\hat \varphi'(x)| \le p | \widetilde \Lambda_n^{(j)} + x (\Lambda_n -\widetilde \Lambda_n^{(j)}) |^{p-2}.
	$$
\end{proof}

The following two lemmas are  simple, but will used  many times in the proof of Theorem~\ref{th: rate of convergence}.
\begin{lemma}
	Assume that for all $p > q \geq 1$ and $a,b > 0$ the following inequality holds
	\begin{equation}\label{appendix eq: inequality for x power p}
	x^p \le a + b x^{q}.
	\end{equation}
	Then
	$$
	x^p \le 2^{\frac{p}{p-q}} (a + b^{\frac{p}{p-q}}).
	$$
\end{lemma}
\begin{proof}
See~\cite{GotzeNauTikh2015a}[Lemma~B.3].
\end{proof}

\begin{lemma}\label{appendix eq: inequality for x power p 2}
	Let $0 < q_1 \le q_2 \le ... \le q_k < p$ and $c_j, j = 0, ... , k$ be positive numbers such that
	$$
	x^p \le c_0 + c_1 x^{q_1} + c_2 x^{q_2} + ... + c_k x^{q_k}.
	$$
	Then
	$$
	x^p \le \beta \left[ c_0 + c_1^{\frac{p}{p-q_1}} + c_2^{\frac{p}{p-q_2}} + ... + c_k^{\frac{p}{p-q_k}} \right],
	$$
	where
	$$
	\beta: = \prod_{\nu=1}^{k} 2^{\frac{p}{p-q_\nu}} \le 2^{\frac{kp}{p-q_k}}.
	$$
\end{lemma}
\begin{proof}
See~\cite{GotzeNauTikh2015a}[Lemma~B.4].
\end{proof}

%\newpage
\bibliographystyle{plain}
\bibliography{literatur}

\end{document}